\documentclass[12pt,reqno]{amsart}
\usepackage{amsthm,amssymb,mathabx}
\usepackage{bm,fullpage}

\usepackage[
  bookmarks = true, 
  pagebackref = false, 
  pdfauthor = {Anderson Cook Kumchev Hughes}, 
  pdftitle = \text{ACHK : Ergodic Waring--Goldbach}, 
  colorlinks = true, 
  linkcolor = blue, 
  citecolor = blue]
{hyperref}

\newtheorem{thm}{Theorem}
\newtheorem*{oldthm}{Theorem}

\newtheorem{lemma}{Lemma}
\newtheorem{prop}{Proposition}

\newtheorem{cor}{Corollary}

\theoremstyle{remark}

\newtheorem{remark}{Remark}
\newtheorem*{acknowledge}{Acknowledgments}

\newcommand{\R}{\ensuremath{\mathbb{R}}}
\newcommand{\T}{\ensuremath{\mathbb{T}}}
\newcommand{\Z}{\ensuremath{\mathbb{Z}}}
\newcommand{\C}{\ensuremath{\mathbb{C}}}
\newcommand{\N}{\ensuremath{\mathbb{N}}}

\newcommand{\Q}{\ensuremath{\mathbf{\mathfrak{f}}}}

\newcommand{\bfa}{\ensuremath{\mathbf{a}}}
\newcommand{\bfq}{\ensuremath{\mathbf{q}}}

\newcommand{\eps}{\ensuremath{\varepsilon}}
\newcommand{\la}{\ensuremath{\lambda}}

\newcommand{\what}{\ensuremath{\widehat{\omega_\la}}}

\newcommand{\eq}{\begin{equation}}
\newcommand{\ee}{\end{equation}}
\newcommand{\p}{\ensuremath{\mathfrak{p}}}

\newcommand{\ZFT}[1]{\widehat{#1}}

\newcommand{\form}{\mathfrak{f}}
\newcommand{\dimension}{n}
\renewcommand{\degree}{k}

\newcommand{\radius}{N} 

\newcommand{\bigradius}{\Lambda}

\newcommand{\vonMangoldtmeasure}{\omega}

\newcommand{\theprimes}{\mathbb{P}}

\newcommand{\apoint}[1]{{\bf #1}}

\newcommand{\unit}{a}
\newcommand{\modulus}{q}
\DeclareMathOperator{\lcm}{lcm}
\newcommand{\Gammakn}[1][{n,k}]{\Gamma_{#1}}

\newcommand{\singseries}{\mathfrak S(\lambda; \mathbf a, \mathbf q)}

\numberwithin{equation}{section}

\title{
On the ergodic Waring--Goldbach problem}
\author[T. Anderson]{Theresa C. Anderson}
\address{
	Department of Mathematics
	\\ University of Wisconsin, Madison
	\\ 480 Lincoln Dr.
	\\ Madison, WI 53705
	\\	U.S.A.
}
\email{tcanderson@math.wisc.edu}

\author[B. Cook]{Brian Cook}
\address{
	Department of Mathematical Sciences
	\\ Kent State University
	\\ Mathematics and Computer Science Building 233 Summit Street
	\\ Kent, OH 44242
	\\	U.S.A.
}
\email{bcook25@kent.edu}

\author[K. Hughes]{Kevin Hughes}
\address{
    School of Mathematics
	\\	The University of Bristol
	\\	Howard House
	\\	Queens Avenue
	\\	Bristol, BS8 1TW
	\\	UK
	\\ and the Heilbronn Insitute for Mathematical Research, Bristol, UK
}
\email{Kevin.Hughes@bristol.ac.uk}

\author[A. Kumchev]{Angel Kumchev}
\address{
	Department of Mathematics
	\\	Towson University
	\\	8000 York Road
	\\	Towson, MD 21252
	\\	U.S.A.
}
\email{akumchev@towson.edu}

\begin{document}
\maketitle

\begin{abstract}
We prove an asymptotic formula for the Fourier transform of the arithmetic surface measure associated to the Waring--Goldbach problem and provide several applications, including bounds for discrete spherical maximal functions along the primes and distribution results such as ergodic theorems. 
\end{abstract}

\section{Introduction}

In this paper, we study several questions on the interface between harmonic analysis and analytic number theory. Our results are motivated in part by the study of discrete maximal functions in harmonic analysis, in part by applications of those maximal functions in ergodic theory, and in part by connections to classical problems in analytic number theory---in particular, the Waring--Goldbach problem.

The harmonic analytic motivation behind our work comes from celebrated results by Bourgain \cite{Bourgain_maximal_ergodic, Bourgain_pointwise_ergodic} on ergodic averages over certain sequences of integers and later work of Magyar, Stein and Wainger \cite{MSW} on discrete spherical maximal functions. Driven by applications in ergodic theory, Bourgain \cite{Bourgain_maximal_ergodic} initiated the study of discrete maximal functions.  A key feature of Bourgain's approach is his use of the circle method from analytic number theory. With this in  mind, Magyar \cite{Magyar_dyadic} provided some partial results on discrete maximal functions related to Waring's problem, leading to Magyar, Stein and Wainger's consideration of the discrete spherical averages
\[ S_\lambda f(\mathbf x) := \frac 1{\#\{ \mathbf y \in \mathbb Z^n : |\mathbf y|_2^2 = \lambda \} } \sum_{|\mathbf y|_2^2 = \lambda} f(\mathbf x - \mathbf y), \]
along with their maximal function
\[ S_*f (\mathbf x) := \sup_{\lambda \in \mathbb N} |S_\lambda f(\mathbf x) |. \]
Here, $f : \mathbb Z^n \to \mathbb C$ and $| \cdot |_2$ denotes the Euclidean norm on $\mathbb R^n$ (thus, $|\mathbf y|_2^2 = y_1^2 + \dots + y_n^2$). 

Magyar, Stein and Wainger \cite{MSW} provided a complete answer to the question of $\ell^p$-bound\-ed\-ness for the maximal operator $S_*$: they proved that, when $n\geq5$, $S_*$ is bounded on $\ell^p(\mathbb Z^n)$ when $p > n/(n-2)$ and unbounded when $p \leq n/(n-2)$. Furthermore, it is shown that this result cannot hold when $n<5$.  In their work, they took the symbiosis between harmonic analysis and number theory a step further by using a full-fledged application of the circle method to analyze the Fourier transform of the arithmetic surface measure underlying the discrete averages $S_\lambda f$. Particularly, define
\[ \widehat{\sigma_\lambda}(\bm\xi) :=  \frac {1}{\#\{\mathbf x \in \mathbb{Z}^n:|\mathbf x|_2^2=\lambda\}}\sum_{|\mathbf x|_2^2=\la} e(\mathbf x\cdot\bm\xi), \]
where $\bm\xi \in \T^n$ and, as usual, $e(z) := e^{2\pi iz}$. Magyar, Stein and Wainger established the following approximation formula for $\widehat{\sigma_\lambda}(\bm\xi)$.

\begin{oldthm}[Magyar--Stein--Wainger]
When $n \ge 5$, one has the decomposition
\begin{equation*}\label{eq:MSW} \widehat{\sigma_\lambda}(\bm\xi) = \sum_{q=1}^\infty \sum_{\substack{ 1 \le a \le q\\ (a,q) = 1}} e(-a\lambda/q) \sum_{\mathbf b \in \Z^n} G(a,q; \mathbf b) \Psi(q\bm\xi - \mathbf b) \widetilde{d\sigma_{{\lambda}}}(\bm\xi-q^{-1}\mathbf b) + \widehat{E_\lambda}(\bm\xi), 
\end{equation*}
where $\widetilde{d\sigma_{{\lambda}}}$ is the continuous Fourier transform of the surface measure of the sphere of radius $\sqrt{\lambda}$, 
\[ G(a,q;\mathbf b) = \sum_{\mathbf x \in (\Z/q\Z)^n} e \left( \frac{a|\mathbf x|_2^2 + \mathbf b \cdot \mathbf x}{q} \right) \]
is an $n$-dimensional Gauss sum, and \( \Psi \) is a smooth bump function which is $1$ on \([-1/8,1/8]^n\) and supported in \([-1/4,1/4]^n\). The convolution operators $E_\lambda$ associated with the error terms $\widehat{E_\lambda}$ satisfy the maximal inequality
\[ \left\| \sup_{\Lambda \leq \lambda \leq 2\Lambda} |E_\la| \right\|_{\ell^2(\Z^n) \to \ell^2(\Z^n)} \lesssim \Lambda ^{1-n/4} \]
for all $\Lambda >0$.   
\end{oldthm}
	
This theorem has served as a model for several authors \cite{Magyar:ergodic,AS,Magyar_discrepancy,Hughes_Vinogradov} who have studied the maximal functions of the discrete surface measures on other arithmetic surfaces over the integers. 
It is also the inspiration for one of the results of the present paper---see Theorem~\ref{thm:approximation_formula} below. 
However, in contrast to earlier work on discrete maximal functions, we study the more singular maximal function of the ``prime points'' on the $k$-sphere. 
The goal of this paper is to study the distribution of points with prime coordinates on the algebraic surface 
\begin{equation}\label{eq1.1}
\form(\mathbf x) := x_1^k+ \dots + x_n^k = \lambda, 
\end{equation} 
for $\lambda \in \mathbb N$. 
By combining number-theoretic techniques from the study of the Waring--Goldbach problem with ideas from harmonic analysis, we are able to prove several results on the distribution of such points, including: an equidistribution theorem, an $L^2$-ergodic theorem, and a pointwise ergodic theorem. A quantitative version of our equidistribution theorem, Theorem \ref{mainmaxfunction} below,  is another main result: we take the spherical maximal function in a new direction by proving $\ell^p(\Z^n)$ bounds for a discrete variant along the primes. 

While restrictions of results about integer sequences to the primes are common in both number theory and ergodic theory, the study of maximal functions related to the primes has been limited to sequential averages (e.g. \cite{Bourgain_approach, Wierdl, Nair, MirekTrojan, MirekTrojanZorin}); this paper appears to be the first work on such restricted problems in harmonic analysis related to prime points on forms in many variables. 
The new obstacles arising in these problems require further development of the Bourgain--Magyar--Stein--Wainger paradigm of using the circle method to decompose the maximal operator and thus to reduce the problem to estimates for exponential sums and integrals. 
Earlier works have been able to employ a classical variant of the circle method which uses the Poisson summation formula to estimate the major arc contribution. The restriction to primes forces us to draw on our knowledge about equidistribution of primes in arithmetic progressions and to employ more primitive tools (compared to Poisson summation) to do so. Therefore, in order to be able to obtain any result at all, we blend mean value theorems of Vinogradov's type into our minor arc analysis. In contrast, in problems over unrestricted integers previous approaches were able to rely merely on $L^\infty$ bounds for the relevant exponential sums. Indeed, insights gained from this present work have already led us to improve (in \cite{ACHK}) on the results of the third author~\cite{Hughes_Vinogradov, Hughes_restricted} on (unrestricted) integer points on the $k$-sphere. 

The study of prime points $\mathbf p = (p_1, \dots, p_n)\in \mathbb P^n$ (here $\mathbb P$ is the set of primes) on the surface \eqref{eq1.1} is known in number theory as the Waring--Goldbach problem. Classic work by Hua \cite{Hua_book} established the asymptotic for the number of representations of a large natural number $\lambda$ as a sum of $n$ $k$th powers of primes when $k$ and $n$ are positive integers such that $n > 2^k$ and $\lambda$ belongs to an appropriate infinite arithmetic progression $\Gammakn$. Write $\log\mathbf x = (\log x_1) \cdots (\log x_n)$, and let $R(\lambda)$ denote the number of prime solutions of \eqref{eq1.1}, counted with logarithmic weights: 
\[ R(\lambda) = \sum_{\form(\mathbf p) = \lambda} \log\mathbf p, \]
where (and through the remainder of the paper) $\mathbf p$ denotes a vector in $\theprimes^n$. Using the Hardy--Littlewood circle method, Hua proved that when $\lambda \to \infty$, one has the asymptotic 
\begin{equation}\label{eq1.2}
R(\lambda) \sim \mathfrak S_{n,k}(\lambda) \lambda^{{n/k - 1}},
\end{equation} 
where $\mathfrak S_{n,k}(\lambda)$ is a product of local densities:
\[  \mathfrak S_{n,k}(\lambda) = \prod_{p \le \infty} \mu_p(\lambda).  \]
Here $\mu_p(\lambda)$ with $p < \infty$ is related to the solubility of \eqref{eq1.1} over the $p$-adic field $\mathbb Q_p$, and $\mu_\infty(\lambda)$ to solubility over the reals. In particular, the set $\Gammakn$ is determined by the requirement that $\mu_p(\lambda) >0$ for all primes $p$. Some examples of progressions $\Gammakn$ (see Chapter VIII in Hua~\cite{Hua_book} for more details, including the full definition of $\Gammakn$) include: 
\begin{itemize}
\item $\Gammakn$ is the residue class $\lambda \equiv n \pmod 2$ when $k$ is odd;
\item $\Gammakn[5,2]$ is the residue class $\lambda \equiv 5 \pmod {24}$;
\item $\Gammakn[17,4]$ is the residue class $\lambda \equiv 17 \pmod {240}$.
\end{itemize}   

The starting point to our main results lies in extending \eqref{eq1.2} to an approximation formula for the Fourier transform of the arithmetic probability measure 
\[ \omega_\lambda({\mathbf x}) := \frac{1}{R(\lambda)}{\bf 1}_{\{{\mathbf p} \in \theprimes^\dimension: \form_{n,k}(\mathbf p)=\lambda\}}({\mathbf x}) \log {\mathbf x}, \]
defined when $R(\lambda)>0$. The Fourier transform of this measure is the exponential sum 
\begin{equation}\label{eq1.3}
\widehat{\omega_\lambda}(\bm\xi) = \frac{1}{R(\lambda)} \sum_{\form(\mathbf p)=\lambda} (\log \mathbf p) e(\mathbf p \cdot \bm\xi).
\end{equation}
We note that $\widehat{\omega_\lambda}$ is defined only for sufficiently large $\lambda\in\Gammakn$ and $n$ sufficiently large in terms of $k$. Based on the current state of affairs in the Waring--Goldbach problem \cite{Kumchev_Wooley1,Kumchev_Wooley2}, the latter means that for large $k$, the value of $n$ must be at least as large as $4k\log k$. In reality, the true size of $R(\la)$ is only known for $n \ge k^2-k + O(\sqrt k)$, so it only makes sense to study the Fourier transform  $\widehat{\omega_\lambda}(\bm\xi)$ when $n \ge k^2 - k$. 

Our first theorem is a variant of the Magyar--Stein--Wainger theorem above for the Fourier transform \eqref{eq1.3}. Before stating the result, we need to introduce some notation. Given an integer $q \ge 1$, we write $\mathbb Z_q = \mathbb Z/q\mathbb Z$ and $\mathbb U_q = \mathbb Z_q^*$, the group of units. If $\mathbf q = (q_1, \dots, q_n) \in \mathbb Z^n$, with $\mathbf q \ge 1$ (by which we mean that $q_i \ge 1$ for all $i$), we write $\mathbb U_{\mathbf q} = \mathbb U_{q_1} \times \dots \times \mathbb U_{q_n}$; it is also convenient to set $\mathbf{a/q} = (a_1/q_1, \dots, a_n/q_n)$ and $\bfa\bfq = (a_1q_1, \dots, a_nq_n)$ if $\mathbf a = (a_1, \dots, a_n)$ is another vector in $\mathbb Z^n$. Given $\lambda \in \mathbb Z$ and $\mathbf a, \mathbf q \in \mathbb Z^n$, with $\mathbf q \ge 1$, we now define 
\begin{gather*} 
g(a,q; b,r) = \frac 1{\varphi([q,r])} \sum_{x \in \mathbb U_{[q,r]}} e\bigg( \frac {ax^k}q + \frac {bx}r \bigg), \\
\singseries = \sum_{q=1}^\infty \sum_{a \in \mathbb U_q} e( -\lambda a/q ) \prod_{i=1}^\dimension g(a, q; a_i, q_i),
\end{gather*}
where $\varphi$ is Euler's totient function and $[q,r] = \lcm[q,r]$. We fix a smooth bump function $\psi$ such that
\[ \mathbf 1_{\mathcal Q}(\mathbf x) \le \psi(\mathbf x) \le \mathbf 1_{\mathcal Q}(\mathbf x/2), \]
where $\mathbf 1_{\mathcal Q}$ is the indicator function of the cube $\mathcal Q = [-1,1]^n$, and we write $\psi_h(\mathbf x) = \psi(h\mathbf x)$ for $h > 0$. We also define $n_1(k) := \min(2^k, k^2 + k) + 3$.

\newcommand{\approximationformula}{Approximation Formula}

\begin{thm}[\approximationformula] \label{thm:approximation_formula}
Let $k \geq 2$ and $n \geq n_1(k)$. Also, let $\lambda \in \Gammakn$ be large, and suppose that $\lambda^{1/k} \le N \lesssim \lambda^{1/k}$. For any fixed $B>0$, there exists a $C = C(B) > 0$ such that one has the decomposition
\begin{equation}\label{eq:approximation_formula}
\what(\bm\xi) = \frac {\lambda^{n/k-1}}{R(\lambda)} \sum_{ 1 \le \bfq \le Q} \sum_{\bfa \in \mathbb U_{\bfq}} \singseries \psi_{N/Q}(\bfq\bm\xi-\bfa) \widetilde{d\sigma_{\lambda}}(\bm\xi-\bfa/\bfq) + \widehat {E_\lambda}(\bm\xi),
\end{equation}
where $Q = (\log N)^C$,  $\widetilde{d\sigma_{\lambda}}$ is the Fourier transform of the $k$-spherical surface measure on the surface defined by \eqref{eq1.1} in $\R_+^n$ $($cf. \eqref{eq:major arcs 3}$)$, and the convolution operators $E_\lambda$ associated with the error terms $\widehat {E_\lambda}(\bm\xi)$ satisfy the maximal inequality
\begin{equation}\label{eq:dyadic_error_bound}
\left\| \sup_{\Lambda \le \lambda \le 2\Lambda} |E_\lambda| \right\|_{\ell^2(\Z^n) \to \ell^2(\Z^n)} \lesssim (\log \Lambda)^{-B}
\end{equation}
for all $\Lambda > 0$. 
\end{thm}

Note that \eqref{eq:dyadic_error_bound} implies that
\begin{equation}\label{eq:sup_error_bound}
\left\| \widehat{E_\lambda} \right\|_{L^\infty(\mathbb T^n)} \lesssim (\log \la)^{-B}. 
\end{equation}
We remark that the proof of Theorem \ref{thm:approximation_formula} allows us to establish \eqref{eq:sup_error_bound} in a slightly wider range of dimension $\dimension$ than the theorem does for the stronger bound \eqref{eq:dyadic_error_bound}. Namely, if $2m$ is any even integer such that one can apply the circle method to establish the asymptotic formula in Waring's problem for $2m$ $k$th powers, then \eqref{eq:sup_error_bound} holds for $n \ge 2m+1$. In particular, using recent advances by Bourgain~\cite{Bourgain_hua} and Wooley \cite{Wooley_asymptotic}, we obtain \eqref{eq:sup_error_bound} for $n \ge n_0(k)$, where $n_0(k) = 2^k+1$ when $k = 2, 3$ or $4$, and
\[ n_0(k) = k^2 + 3 - \max_{1 \le j \le k-1} \left\lceil \frac {kj - \min(2^j,j^2+j)}{k-j+1} \right\rceil \]
when $k \ge 5$. These observations are useful in our next result, which describes the decay of $\widehat{\omega_\la}$ at irrational frequencies.

\begin{thm}\label{thm:decayatirrationalfrequencies}
Let $\degree \geq 2$ and $\dimension \ge \dimension_0(\degree)$. If $\bm\xi \not\in \mathbb{Q}^\dimension$, then $\what(\bm\xi)\to0$ as $\lambda \to \infty$ along $\Gammakn$. 
\end{thm} 

Let $r(\lambda)$ denote the number of prime points on the $\degree$-sphere \eqref{eq1.1}. It follows readily from Theorem \ref{thm:decayatirrationalfrequencies} that, when $\bm\xi \not\in \mathbb{Q}^\dimension$, one has 
\begin{equation}\label{WeylCriterion}
\lim_{\substack{\lambda \to\infty \\ \lambda \in \Gammakn}} \frac{1}{r(\lambda)} \sum_{\form(\mathbf p)=\lambda} e(\mathbf p \cdot \bm\xi) = 0.
\end{equation}  
This gives a pair of interesting corollaries. The first is obtained by noting that \eqref{WeylCriterion} is precisely the Weyl criterion for uniform distribution on a torus. 
				
\begin{cor}
Let $\degree \geq 2$, $\dimension \ge \dimension_0(\degree)$, and  $\bm\alpha \in (\R\setminus\mathbb{Q})^\dimension$. The sets 
\[ \{(\alpha_1p_1, \dots, \alpha_\dimension p_\dimension) : \form(\mathbf p) = \la \} \]
become uniformly distributed with respect to the Lebesgue measure on the $\dimension$-dimensional torus $\T^\dimension $ as $\lambda \to\infty$ along  $\Gammakn$.
\end{cor} 
					
Our second corollary is an $L^2$-convergence result regarding certain ergodic averages; as in Section~4 of \cite{Magyar:ergodic}, where the analogous `integral' result is proven, this follows from the spectral theorem for unitary operators. To state this corollary, let $(X,\mu)$ denote a probability space with a commuting family of $\dimension$ invertible measure preserving transformations $T=(T_1,...,T_\dimension)$. Such a family is referred to as a \textit{fully ergodic family of transformations} if the hypothesis 
\[ T^s_1f=T^s_2f= \dots =T^s_nf=f,\]
where $s \in \N$ and $f\in L^2(X,\mu)$, implies that $f$ is constant. Here $Tf$ should be interpreted as $f\circ T$. As observed in \cite{Magyar:ergodic}, the notion of full ergodicity is  actually a condition on the joint spectrum of the $T_i$. More precisely, full ergodicity implies that given $f\in L^2(X,\mu)$, if $T_if=e(\lambda_i)f$ holds with $\lambda_i$ rational for all $i\leq n$, then $f$ is constant almost everywhere.  For a function $f: X \to \C$, $\lambda \in \Gammakn$ and $x \in X$, define the Waring--Goldbach ergodic averages on $X$ with respect to $T$  by 
\begin{equation}\label{ergavg}
\mathcal{A}_\lambda f(x) := \frac{1}{R(\lambda)} \sum_{\form({\bf p})=\lambda} (\log{\bf p}) f(T^{\bf p}x), 
\end{equation}
where $T^{\bf m}x := T_1^{m_1} \cdots T_\dimension^{m_\dimension} x$ for ${\bf m} = (m_1, \dots, m_\dimension) \in \Z^\dimension$. 

\begin{cor}[$L^2$-mean ergodic theorem]
Let $k \ge 2$,  $\dimension \ge \dimension_0(\degree)$, and let $(X,\mu)$ be a probability space with a fully ergodic family of transformations $T=(T_1,...,T_n)$. Then for all \( f \in L^2(X,\mu) \), the ergodic averages of $f$ defined by \eqref{ergavg} converge in $L^2(X, \mu)$ to the space average of $f$; that is, one has that 
\[ \lim_{\substack{\lambda \to\infty \\ \lambda \in \Gammakn}} \mathcal{A}_\la f =\int_X f \, d\mu \]
in $L^2(X, \mu)$. 
\end{cor}

\begin{remark}
As observed in Section 3 of \cite{Magyar:ergodic}, this result does not hold in general if one omits the full ergodicity condition.
\end{remark}

To prove the ergodic theorems, we consider the convolution operator $A_\lambda$ with Fourier multiplier $\widehat{\omega_\lambda}$: for functions $f: \Z^\dimension \to \C$, we write
\begin{equation}\label{primeavg} 
A_\lambda f := \omega_\lambda \star f. 
\end{equation}
This is our discrete spherical averaging operator along the primes. We will use the \approximationformula\ to prove a maximal theorem, stated below. In the remaining theorems, define $\dimension_2(\degree) := \degree^2(\degree-1)+1$ for $\degree \geq 7$ and $\dimension_2(\degree) := \degree 2^{\degree-1}+1$ for $2 \leq \degree \leq 6$; also define $p_{\degree,\dimension} := 1+\frac{\dimension_2(\degree)}{2\dimension-\dimension_2(\degree)} = \frac{2\dimension}{2\dimension-\dimension_2(\degree)}$. 

\begin{thm}\label{mainmaxfunction}
Let $\degree \geq 2$ and $\dimension \geq \max\{ \dimension_1(\degree), \dimension_2(\degree) \}$. The maximal function given by 
\begin{equation}
A_* f := \sup_{\lambda \in \Gammakn} |A_\lambda f| 
\end{equation} 
is bounded on \(\ell^p(\Z^\dimension)\) for all $p > p_{\degree,\dimension}$. 
\end{thm}

\begin{remark}
In sufficiently large dimensions, the maximal function $A_*$ is unbounded on \(\ell^p(\Z^n)\) for \(p < \frac{\dimension}{\dimension-\degree}\). This can readily be seen by testing the maximal function on a delta function at the origin and using the asymptotic for $R(\lambda)$ as $\lambda \to \infty$ in $\Gammakn$. With this in mind, we conjecture that \(A_*\) should be bounded on \(\ell^p(\Z^\dimension)\) for all \(p>\frac{n}{n-k}\) in sufficiently large dimensions; this is the same conjectured range of $p$ as for the integral maximal function. We refer the reader to \cite{Hughes_Vinogradov} for more information on the conjectured range of $\ell^p(\Z^\dimension)$-boundedness for the integral maximal function.
\end{remark}

\begin{remark}
In the quadratic case, the Magyar--Stein--Wainger theorem holds for $\dimension \geq 5$ whereas ours only holds for $\dimension \geq 7$. (Theorem~\ref{mainmaxfunction} does match the Magyar--Stein--Wainger theorem in the range of $p$, and both ranges are sharp.) 
An aspect of this work is that for improvements to the value of dimension and $p_{\degree,\dimension}$ in the integer setting automatically translate to corresponding improvements to $\dimension_2(\degree)$ and $p_{\degree,\dimension}$ in our setting. We used our techniques to improve the range of dimension and $p_{\degree,\dimension}$ in the integer setting when the degree $\degree$ is sufficiently large in a forthcoming paper \cite{ACHK}. 
\end{remark}

We take this moment to describe the proof of our maximal theorem and to compare it with previous works. 
Throughout the paper we follow the paradigms of \cite{Bourgain_maximal_ergodic} as embellished in the integral version of our averages in  \cite{MSW} and \cite{Magyar:ergodic}. In particular we assume that the reader is familiar with the transference technology of \cite{MSW}. As in \cite{MSW}, our maximal theorem will exploit the {\approximationformula} which decomposes $\what = \widehat{M_{\lambda}}+\widehat{E_{\lambda}}$ into the sum of a main term and error term. We will use separate techniques to get good bounds on the suprema over $\lambda$ of both the main term and error term. As in all previous works, our decomposition requires a major arc/minor arc decomposition of the degree $k$ frequency variable. 
Unlike previous works we require an additional major arc/minor arc decomposition of the linear frequency variables. 
For the main term we will use estimates for relevant exponential sums and oscillatory integrals in addition to the transference results of \cite{MSW} to bound the main term. 
However, as already mentioned, the methods in previous works such as  \cite{MSW,Hughes_Vinogradov,Hughes_restricted} are insufficient to handle the error term from our circle method approximation in the \approximationformula. This is due to the logarithmic decay in \eqref{eq:dyadic_error_bound} as opposed to power savings that appeared in previous works. To overcome this obstacle, we introduce a \emph{hybrid sup and mean value bound} to control the relevant exponential sums on our set of minor arcs and consequently bound the error term in $\ell^2$; this is one of the novel aspects of our paper. From this, the known bounds for the integer case in \cite{MSW}, and the bounds we are able to proe for the main term on $\ell^p$, we are able to bound the analogue of the Magyar--Stein--Wainger discrete spherical maximal function along the primes.  

Following Magyar \cite{Magyar:proc} and Bourgain \cite{Bourgain_maximal_ergodic}, we will use our maximal theorem to prove the following pointwise ergodic theorem along the primes. 

\begin{thm}\label{thm:pointwise_ergodic}
Let $\degree \geq 2$, $\dimension \geq \max\{ \dimension_1(\degree), \dimension_2(\degree) \}$, and let \( (X,\mu) \) be a probability space with a fully ergodic family of transformations $T=(T_1,...,T_n)$. 
Then for all \( f \in L^2(X,\mu) \), the ergodic averages of $f$ defined by \eqref{ergavg} converge almost everywhere to the space average of $f$; that is, 
\begin{equation}
\lim_{\substack{\lambda \to\infty \\ \lambda \in \Gammakn}} \mathcal{A}_\lambda f = \int_X f \, d\mu 
\end{equation} 
\(\mu\)-almost everywhere. 
\end{thm}

Again, a standard argument (see for instance \cite{Wierdl}) implies the same result without the logarithmic weights. 

\begin{cor}
Suppose that \( (X,\mu) \) is a probability space with \(\dimension\) commuting measure-preserving operators \( T_1,\dots,T_n \) satisfying the conditions of Theorem \ref{thm:pointwise_ergodic}. Then, for all $f \in L^2(X,\mu)$, one has
\begin{equation}
\lim_{\substack{\lambda \to\infty \\ \lambda \in \Gammakn}} \frac{1}{r(\lambda)} \sum_{\form({\bf p}) = \lambda} f(T^{\bf p}x) =  \int_X f \, d\mu 
\end{equation}
\(\mu\)-almost everywhere. 
\end{cor}

Combining our pointwise ergodic theorem on \(\ell^2\) with our maximal function bounds, we immediately obtain, via standard approximation arguments, the following corollary. 

\begin{cor}
Suppose that \( (X,\mu) \) is a probability space with \(\dimension\) commuting measure-preserving transformations \( T_1,\dots,T_n \) as in Theorem \ref{thm:pointwise_ergodic}. Then, for \(p>p_{k,n}\) and for all $f \in L^p(X,\mu)$, one has 
\begin{equation}
\lim_{\substack{\lambda \to\infty \\ \lambda \in \Gammakn}} \mathcal{A}_\lambda f = \int_X f \, d\mu 
\end{equation} 
\(\mu\)-almost everywhere. 
\end{cor}
						
The paper is organized as follows. In Section \ref{sec2}, we collect some needed number theoretic facts. Then in Section \ref{sec3}, we use the circle method to decompose $\widehat{\omega_\lambda}$ into a main term and an error term; we also prove $\ell^2$ bounds on the error in this section. One key additional technical difficulty here compared with the work in \cite{MSW} is that the precise shape of our error terms is more complicated than in the integral case; in particular, we need to perform a major and minor arc analysis of the linear phases (in addition to the higher degree phases). In Section \ref{sec:main_term}, we develop a careful analysis and interpolation argument to get $\ell^p$ bounds on the main term, since we cannot apply the techniques in \cite{MSW} directly. In Section~\ref{section:integral_comparison}, we compare the averages along the primes to the integral ones to control the error terms and prove Theorem~\ref{mainmaxfunction}. Finally, we prove the ergodic theorems in Section \ref{sec:6}.

\begin{acknowledge}
The first author was supported by NSF grant DMS-1502464. Parts of this work were done while the first author was in residence at MSRI, Berkeley in Spring 2017 and while the fourth author was visiting the University of Bristol with support from the ERC Advanced Grant No. 695223. The second author was supported by NSF grant DMS-1147523 and by the Fields Institute. He would also like to thank Tim Khaner and the University of Alberta's Department of Anthropology for being such gracious hosts during the Summer of 2015.
\end{acknowledge}

\newcommand{\quadraticmajorarc}{\mathfrak{M}^{\unit/\modulus}}
\newcommand{\quadraticmajorarcs}{\mathfrak{M}}
\newcommand{\quadraticminorarcs}{\mathfrak{m}}

\newcommand{\linearmajorarc}{{\mathfrak{N}}^{\apoint{\unit}/\apoint{\modulus}}}
\newcommand{\linearmajorarcs}{{\mathfrak{N}}}
\newcommand{\linearminorarcs}{{\mathfrak{n}}}
\newcommand{\majorarcsmollifier}{\Psi_{\linearmajorarcs}}

\section{Bounds for exponential sums and integrals}
\label{sec2}

Here we recall and prove some results from analytic number theory.

\begin{lemma}\label{lem:Shparlinski}
Let $a, b, q$ be integers with $\gcd(a, b, q) = 1$. Then, for any fixed $\eps > 0$, one has
\[
  \sum_{x \in \mathbb U_q} e\bigg( \frac {ax^k + bx}q \bigg) \lesssim q^{1/2+\eps}.
\]
\end{lemma}

\begin{proof}
This is a special case of Theorem 1 of Shparlinski \cite{shpar96}.
\end{proof}

\begin{lemma}\label{lem:Harman}
Let $f(x) = \alpha x^k + \dots + \alpha_1x\in \mathbb R[x]$, with $k \ge 2$, and suppose that there exist integers $a,q$ such that $(a,q) = 1$ and $|q\alpha - a| \le q^{-1}$. Then
\[ \sum_{p \le N} (\log p) e(f(p)) \lesssim NL^c \big( q^{-1} + N^{-1/2} + qN^{-k} \big)^{2^{1-2k}}, \]
where $L = \log N$ and $c = c_k$ is a constant. 
\end{lemma}

\begin{proof}
This is a variant of Theorem 1 in Harman \cite{Harman:exp_sums}, where the exponent of $2^{1-2k}$ is replaced by $4^{1-k}$ at the expense of replacing the factor $L^c$ above by $N^{\eps}$. The present version is well-known to the experts, but since we were unable to locate it in the literature, we will provide a brief sketch of the argument. The proof requires small adjustments to the proofs of Lemmas~2--4 in \cite{Harman:exp_sums}. Those proofs use the inequality
\begin{equation}\label{eq2.1}
\sum_{x \le X} \tau_r(x) \min\big( Y, \| \theta x \|^{-1} \big) \lesssim X^{\eps} \sum_{x \le X} \min\big( Y, \| \theta x \|^{-1} \big),
\end{equation}
where $\tau_r(x)$ is the $r$-fold divisor function. However, in most places the above inequality is used for convenience rather than by necessity. The places where this inequality is really needed occur towards the ends of the proofs of Lemmas 3 and 4 in \cite{Harman:exp_sums}, when one wants to apply a standard estimate (e.g., Lemma 2.2 in Vaughan~\cite{Vaughan}) to the sum on the right side of \eqref{eq2.1}. In those places, we can replace \eqref{eq2.1} with
\[ \sum_{x \le X} \tau_r(x) \min\big( Y, \| \theta x \|^{-1} \big) \lesssim (XY)^{1/2}(\log X)^c \bigg\{ \sum_{x \le X} \min\big( Y, \| \theta x \|^{-1} \big) \bigg\}^{1/2}. \]
We can then follow the rest of Harman's proof.
\end{proof}

\begin{lemma}\label{lem:Balog}
Let $a,b,q,r$, be integers such that $(a,q) = (b,r) = 1$ and $|\alpha - a/q| \le 2N^{-1}$. Then
\[ \sum_{\substack{p \le N\\p \equiv b \!\!\!\! \pmod r}} (\log p) e(\alpha p) \lesssim NL^3 \big( q^{-1} + N^{-2/5} + qN^{-1} \big)^{1/2}. \] 
\end{lemma}

\begin{proof}
This is the main result of Balog and Perelli \cite{BalogPerelli}, with some of the terms slightly simplified for use in the present context.
\end{proof}

When $1 \le Q \le X$, we define the \emph{set of major arcs} $\mathfrak{M}(X, Q)$ by
\[ \mathfrak M(X,Q) = \bigcup_{q \le Q} \bigcup_{a \in \mathbb U_q} \big\{ \theta \in \mathbb T : |q\theta - a| \le QX^{-1} \big\}. \]
The complement of a set of major arcs, $\mathfrak{m}(X,Q) = \mathbb T \setminus \mathfrak M(X,Q)$, is the respective \emph{set of minor arcs}. When working with a particular choice of major and minor arcs, we may write $\mathfrak{M}_{a/q}$ for the major arc centered at the rational $a/q$. Note that when $2Q < X$, the set $\mathfrak{M}(X, Q)$ is the disjoint union of closed intervals of total measure $O(QX^{-1})$. \\

Our analysis of $\widehat{\omega_\lambda}(\bm\xi)$ will depend on the exponential sum
\begin{equation}\label{eq2.S} 
S_N(\theta,\xi) = \sum_{p \le N} (\log p)e(\theta p^k + \xi p), 
\end{equation}
where the summation is over the prime numbers $p \le N$. In particular, we need to approximate $S_N(\theta, \xi)$ when both $\theta$ and $\xi$ are near rationals with small denominators. The approximations involve the exponential sum $g(a,q; b,r)$ defined above and the oscillatory integral
\begin{equation}\label{eq2.I} 
 I_N(\delta, \eta) = \int_0^N e\big( \delta x^k + \eta x \big) \, dx.  
\end{equation}
We note that, by Lemma \ref{lem:Shparlinski},  
\begin{equation}\label{eq2.3}
g(a,q; b,r) \lesssim [q,r]^{-1/2 + \eps},
\end{equation}
and that the $k$th-derivative estimate for oscillatory integrals (Proposition 2 on p. 332 in Stein~\cite{Stein93}) yields
\begin{equation}\label{eq2.4}
I_N(\delta, \eta) \lesssim \frac N{(1 + N^k|\delta|) ^{1/k}}.
\end{equation}
Furthermore, since 
\[ I_N(\delta, \eta) = k^{-1}N\int_0^1 u^{1/k-1} e\big( \delta_0 u + \eta_0 u^{1/k} \big) \, du, \]
where $\delta_0 = \delta N^k$ and $\eta_0 = \eta N$, we can also apply the second-derivative estimate (the case $k=2$ of the corollary on p. 334 of \cite{Stein93}) to deduce the bound
\begin{equation}\label{eq2.5}
I_N(\delta, \eta) \lesssim \frac N{(1 + N|\eta|) ^{1/2}}.
\end{equation}
Our next lemma uses the Siegel--Walfisz theorem to approximate $S_N(\theta, \xi)$.

\begin{lemma}\label{lem:major arc approx}
Let $Q, R \le (\log N)^C$ for some fixed $C > 0$, let $\theta \in \mathfrak{M}_{a/q}$ for some major arc of the set $\mathfrak M = \mathfrak M(N^k, Q)$, and let $\xi \in \mathfrak N_{b/r}$ for some major arc of the set $\mathfrak N = \mathfrak M(N, R)$. Then
\[ S_N(\theta, \xi) = g(a,q; b,r)I_N(\theta - a/q, \xi - b/r) + O\big( N(QR)^{-10} \big). \] 
\end{lemma}

\begin{proof}
We write $\delta = \theta - a/q$, $\eta = \xi - b/r$, and $s = [q,r]$. When we partition the exponential sum $S_N(\theta,\xi_j)$ into sums over primes in fixed arithmetic progressions, we find that 
\begin{align}
S_N(\theta,\xi) &= \sum_{h \in \mathbb U_s} \sum_{\substack{p \le N\\ p \equiv h \!\!\!\! \pmod s}} (\log p) e\bigg( \bigg( \frac aq + \delta \bigg)p^k + \bigg( \frac br + \eta \bigg) p \bigg) + O(s) \notag\\
&= \sum_{h \in \mathbb U_s} e\bigg( \frac {ah^k}q + \frac {bh}{r} \bigg) \sum_{\substack{p \le N\\ p \equiv h \!\!\!\! \pmod s}} (\log p) e( \delta p^k + \eta p ) + O\big( QR \big). \label{eq2.6}   
\end{align}
Since $s \le QR \le (\log N)^{2C}$ and $h \in \mathbb U_s$, the Siegel--Walfisz theorem yields  
\[ \sum_{\substack{p \le x\\ p \equiv h \!\!\!\! \pmod s}} \log p = \frac {x}{\varphi(s)} + O\big( N(QR)^{-12} \big) \]
for all $x \le N$. Using  this asymptotic formula and partial summation, we obtain
\begin{equation}\label{eq2.7}
\sum_{\substack{p \le N\\ p \equiv h \!\!\!\! \pmod s}} (\log p) e( \delta p^k + \eta p ) = \varphi(s)^{-1}I_N(\delta,\eta) + O\big( N(QR)^{-11} \big).
\end{equation}
The lemma follows from \eqref{eq2.6} and \eqref{eq2.7}. 
\end{proof}

\begin{lemma}\label{lem:maxBourgain}
Let $k \ge 2$ and $2s \ge \min(2^k, k^2+k) + 2$. Then
\begin{equation}\label{eq:lem6.0}
\int_{\T} \sup_{\xi} |S_N(\theta, \xi) |^{2s} \, d\theta \lesssim N^{2s-k}L^{2s},
\end{equation}
where $L = \log N$.
\end{lemma}

\begin{proof}
Set $H_j = sN^j$ and define 
\[ a_h(\theta) = \sum_{\substack{p_1, \dots, p_s \le N\\ p_1 + \dots + p_s = h}} (\log\mathbf p) e(\theta \form_{s,k}(\mathbf p)), \]
so that
\[ S_N(\theta, \xi)^s = \sum_{h \le H_1} a_h(\theta)e(\xi h). \]
By applying Cauchy's inequality, we deduce that
\[ \sup_{\xi} |S_N(\theta, \xi)|^{2s} \le H_1\sum_{h \le H_1} |a_h(\theta)|^2. \]
Hence, 
\begin{equation}\label{eq:lem6.1}
\int_{\T} \sup_{\xi} |S_N(\theta, \xi) |^{2s} \, d\theta \le H_1\sum_{h \le H_1} \int_\T a_h(\theta)\overline{a_h(\theta)} \, d\theta.
\end{equation} 

By orthogonality, 
\[ \int_\T a_h(\theta)\overline{a_h(\theta)} \, d\theta = \sum_{\mathbf p, \mathbf p' : \eqref{eq:lem6.2}} (\log\mathbf p) (\log \mathbf p'), \]
where $\mathbf p, \mathbf p' \le N$ and satisfy the conditions
\begin{equation}\label{eq:lem6.2}
\form_{s,k}(\mathbf p) = \form_{s,k}(\mathbf p'), \quad \form_{s,1}(\mathbf p) = \form_{s,1}(\mathbf p') = h.
\end{equation}
Thus,
\begin{equation}\label{eq:lem6.1a} 
\sum_{h \le H_1} \int_\T a_h(\theta)\overline{a_h(\theta)} \, d\theta \lesssim L^{2s}I_{s,k}(N), 
\end{equation} 
where $I_{s,k}(N)$ denotes the number of integer solutions of the system
\begin{equation}\label{eq:lem6.3}
\form_{s,k}(\mathbf x) = \form_{s,k}(\mathbf y), \quad \form_{s,1}(\mathbf x) = \form_{s,1}(\mathbf y), 
\end{equation}
with $1 \le \mathbf x, \mathbf y \le N$. The lemma follows from \eqref{eq:lem6.1}, \eqref{eq:lem6.1a}
and the inequality
\begin{equation}\label{eq:lem6.4} 
I_{s,k}(N) \lesssim N^{2s-k-1}, 
\end{equation}
which we establish next.

Under the hypothesis $2s \ge 2^k+2$, the bound \eqref{eq:lem6.4} is a direct consequence of the main result of Br\"udern and Robert \cite{Brudern}. On the other hand, by grouping the solutions of \eqref{eq:lem6.3} according to the values of the expressions $\form_{s,j}(\mathbf x) - \form_{s,j}(\mathbf y)$, $1 < j < k$, we find that
\begin{equation}\label{eq:lem6.5} 
I_{s,k}(N) \le \sum_{|h_2| < H_2} \cdots \sum_{|h_{k-1}| < H_{k-1}} J_{s,k}(N; 0, h_2, \dots, h_{k-1}, 0), 
\end{equation}
where $J_{s,k}(N; \mathbf h)$ is the generalized Vinogradov integral
\[ J_{s,k}(N; \mathbf h) = \int_{\T^k} \bigg| \sum_{x \le N} e\big( \alpha_kx^k + \dots + \alpha_1x \big) \bigg|^{2s} e(-\bm\alpha \cdot \mathbf h) \, d\bm\alpha. \]
We can now refer to the recent work by Bourgain, Demeter and Guth \cite{BourgainDemeterGuth} on the classic Vinogradov integral $J_{s,k}(N) = J_{s,k}(N; \mathbf 0)$ to get
\[ J_{s,k}(N; \mathbf h) \le J_{s,k}(N) \lesssim  N^{2s-k(k+1)/2}, \]
provided that $2s > k(k+1)$ (see \S5 in \cite{BourgainDemeterGuth}). Inserting this bound into the right side of \eqref{eq:lem6.5} yields \eqref{eq:lem6.4} for $2s > k^2+k$. 
\end{proof}

In \S\ref{sec:main_term}, we will need some more refined estimates for $g(a,q;b,r)$ and its averages; we establish those in the next lemma. Here, $\mu(n)$ denotes the M\"obius function from number theory (see \S16.3 in Hardy and Wright \cite{HardyWright}). 

\begin{lemma}\label{lem:g-sum}
Let $a,b,q,r$, be integers with $(a,q) = (b,r) = 1$, and write $q_0 = q/(q,r)$ and $r_0 = r/(q,r)$. Then:
\begin{itemize}
\item [(i)] if $(r_0, q) > 1$, one has $g(a,q; b,r) = 0$; 
\item [(ii)] if $(r_0, q) = 1$, one has 
\[ g(a, q; b, r) = \frac {\mu(r_0)}{\varphi(r_0)} g(ar_0^k,q; bq_0,q); \] 
\item [(iii)] one has
\begin{equation}\label{eq:g-sum.iii}
\sum_{u \in \Z_r} \bigg| \sum_{b \in \mathbb U_r} g(a, q; b, r)e(-ub/r) \bigg| \le \frac {\tau(r)r}{\varphi(r_0)}.   
\end{equation}
\end{itemize}
\end{lemma}

\begin{proof}
(i) Suppose that $(r_0, q) > 1$. Then there is a prime number $p$ and positive integers $\alpha, \beta$, with $\alpha < \beta$, such that
\[ p^\alpha \mid q, \quad p^{\alpha+1} \nmid q, \quad p^\beta \mid r, \quad p^{\beta+1} \nmid r. \]
Let $q = p^\alpha q_1$ and $r = p^\beta r_1$. By a change of the summation variable $x \in \mathbb U_{[q,r]}$ in $g(a,q; b,r)$ to $x = p^\beta y + [q_1,r_1]z$, where $y \in \mathbb U_{[q_1,r_1]}$ and $z \in \mathbb U_{p^\beta}$, we can factor the exponential sum $g(a,q; b,r)$ as
\begin{equation}\label{eq2.8}
g(a,q; b,r) = g(ap^{k\beta - \alpha}, q_1; b, r_1)g(a_1, p^\alpha; b_1, p^\beta), 
\end{equation}
where $a_1 = a[q_1,r_1]^kq_1^{-1}$ and $b_1 = b[q_1,r_1]r_1^{-1}$. We note that $(a_1,p) = (b_1,p) = 1$. Next, we write the variable $z \in \mathbb U_{p^\beta}$ in $g(a_1, p^\alpha; b_1, p^\beta)$ as $z = u + p^\alpha v$, where $u \in \mathbb U_{p^\alpha}$ and $v \in \mathbb Z_{p^{\gamma}}$, $\gamma = \beta-\alpha$. This gives
\[ \varphi(p^{\beta})g(a_1, p^\alpha; b_1, p^\beta) = \sum_{u \in \mathbb U_{p^\alpha}} e\bigg( \frac {a_1u^k}{p^\alpha} + \frac {b_1u}{p^\beta} \bigg) \sum_{v \in \mathbb Z_{p^{\gamma}}} e\bigg( \frac {b_1v}{p^{\gamma}} \bigg). \]
Since $(b_1,p) = 1$, the last sum over $v$ vanishes. Together with the factorization \eqref{eq2.8}, this proves (i). 

(ii) When $(q,r_0) = 1$, we change the summation variable $x \in \mathbb U_{[q,r]}$ in $g(a,q; b,r)$ to $x = r_0 y + qz$, where $y \in \mathbb U_q$ and $z \in \mathbb U_{r_0}$. Similarly to \eqref{eq2.8}, we have
\[ g(a,q; b,r) = g(ar_0^k, q; b, (q,r))\varphi(r_0)^{-1}\sum_{z \in \mathbb U_{r_0}} e\bigg( \frac {bq_0z}{r_0} \bigg). \]
We now note that the last exponential sum is a Ramanujan sum modulo $r_0$ and $(bq_0,r_0) = 1$. Hence, the claim follows from a classical expression for the Ramanujan sum (see Theorem~272 in Hardy and Wright \cite{HardyWright}).

(iii) Let $h(a,q; u,r)$ denote the sum over $b$ on the left side of \eqref{eq:g-sum.iii}. By part (i), we may assume that $(q,r_0) = 1$. We can then use part (ii) to rewrite $h(a,q; u,r)$ as 
\[ h(a,q; u,r) = \frac {\mu(r_0)}{\varphi(r_0)\varphi(q)} \sum_{x \in \mathbb U_q} e\bigg( \frac {ar_0^kx^k}q \bigg) \sum_{b \in \mathbb U_r} e \bigg( \frac {(r_0x - u)b}r \bigg). \]
Since the inner sum is a Ramanujan sum, we deduce that
\[ |h(a,q; u,r)| \le \frac {1}{\varphi(r_0)\varphi(q)} \sum_{d \mid r} \, d \!\!\! \sum_{\substack{x \in  \mathbb U_q\\ d \mid (r_0x - u)}} \!\!\!\! 1. \]
We remark that a divisor $d$ of $r$ factors uniquely as $d = d_1d_2$, where $d_1 \mid (q,r)$ and $d_2 \mid r_0$. When $d_2 \nmid u$, the sum over $x$ vanishes. On the other hand, when $d_2 \mid u$, the condition $d \mid (r_0x - u)$ restricts $h$ to a single residue class modulo $d_1$; hence, the inner sum is then bounded by $\varphi(q)/d_1$. We conclude that
\[ |h(a,q; u,r)| \le \frac {1}{\varphi(r_0)\varphi(q)} \sum_{d_1 \mid (q,r)}\sum_{d_2 \mid (r_0,u)} d_1d_2 \bigg(  \frac {\varphi(q)}{d_1} \bigg) = \frac {\tau((q,r))}{\varphi(r_0)} \sum_{d \mid (r_0,u)} d. \]
Summing the last bound over $u$, we deduce
\[ \sum_{u \in \Z_r} |h(a,q; u,r)| \le \frac {\tau((q,r))}{\varphi(r_0)} \sum_{u \in \Z_r} \sum_{d \mid (r_0,u)} d = \frac {\tau((q,r))}{\varphi(r_0)} \sum_{d \mid r_0} d \sum_{\substack{u \in \Z_r\\ d \mid u}} 1 = \frac {\tau(r)r}{\varphi(r_0)}, \]
where we have used that $\tau((q,r))\tau(r_0) = \tau(r)$. 
\end{proof}

\section{Proof of the \approximationformula}
\label{sec3}

In this section, we use the circle method to prove Theorem \ref{thm:approximation_formula}. However, before we proceed with that, we establish a lemma that allows us to leverage our estimates for exponential sums to bound various dyadic maximal functions, including the maximal function of the error term. 

\begin{lemma}\label{lem:L2 bound}
Let $\mathcal L$ be a set of integers. For $\lambda \in \mathcal L$, let $T_\lambda$ be a convolution operator on $\ell^2(\mathbb Z^d)$ with Fourier multiplier $\widehat{m_\lambda}(\bm\xi)$ given by
\[ \widehat{m_\lambda}(\bm\xi) = \int_{X} K(\theta; \bm\xi) e(\Phi(\lambda,\theta)) \, d\mu(\theta), \]
where $(X, \mu)$ is a measure space, $\Phi: \Z \times X \to \R$, and $K( \cdot; \bm\xi) \in L^1(X, \mu)$ is a kernel independent of $\lambda$. Let
\[ (T_*f)(\mathbf x) = \sup_{\lambda \in \mathcal L} |(T_\lambda f)(\mathbf x)|. \]
Then 
\[ \| T_* \|_{\ell^2(\mathbb Z^d) \to \ell^2(\mathbb Z^d)} \le \int_{X} \sup_{\bm\xi \in \mathbb T^d}  | K( \theta; \bm\xi) | \, d\mu(\theta). \]
\end{lemma}

\begin{proof}
Suppose that $f \in \ell^2(\Z^d)$. We first exchange the order of integration to get 
\begin{align*} 
|(T_\lambda f)(\mathbf x)| &= \bigg| \int_{\T^d} \int_{X} K(\theta; \bm\xi) \widehat{f}(\bm\xi) e(\Phi(\lambda,\theta) - \mathbf x \cdot \bm\xi) \, d\mu(\theta) d\bm\xi \bigg| \\
&\leq \int_{X} \bigg| \int_{\T^d} K(\theta;\bm\xi) \widehat{f}(\bm\xi) e(-\mathbf x \cdot \bm\xi) \, d\bm\xi \bigg| \, d\theta =: \int_{X} |g(\theta; \mathbf x)| \, d\mu(\theta). 
\end{align*}
Note that since the last integral is independent of $\lambda$, the same bound holds for $(T_*f)(\mathbf x)$. Consequently,
\begin{align*}
\| T_*f \|_{\ell^2(\Z^d)} &\le \bigg\| \int_{X} |g(\theta; \cdot)| \, d\mu(\theta) \bigg\|_{\ell^2(\Z^d)} \\
&\le  \int_{X}  \bigg\{ \sum_{\mathbf x \in \mathbb Z^d} |g(\theta; \mathbf x)|^2 \bigg\}^{1/2} \, d\mu(\theta) \\
&\le  \int_{X} \bigg\{ \int_{\T^d} | K(\theta;\bm\xi) \widehat{f}(\bm\xi) |^2 \, d\bm\xi \bigg\}^{1/2} \, d\mu(\theta) \\
&\le \int_{X} \sup_{\bm\xi \in \T^d} | K(\theta;\bm\xi) | \big \| \widehat f \big\|_{L^2(\T^d)} \, d\mu(\theta),
\end{align*}
on using Minkowski's and Bessel's inequalities. The lemma follows by applying Plancherel's theorem to $f$ and $\hat f$.
\end{proof}

Let $\lambda \in \Gammakn \cap [\Lambda, 2\Lambda]$. Suppose that $N^k \ge \lambda$ and write $L = \log N$. By orthogonality,
\begin{align}
R(\lambda)\ZFT{\omega_\lambda}(\bm\xi) 
&= \sum_{1 \le \mathbf p \le N} (\log\mathbf p)e(\mathbf p \cdot \bm\xi) \int_{\mathbb T} e([\form(\mathbf p) - \lambda]\theta) \, d\theta \notag\\
&= \int_{\mathbb T} \bigg\{ \prod_{j=1}^n S_N(\theta, \xi_j)  \bigg\} e(-\lambda \theta) \, d\theta =: \int_{\mathbb T} F(\theta; \bm\xi) e(-\lambda \theta) \, d\theta, \label{eq:HL integral} 
\end{align} 
where $S_N(\theta, \xi)$ is the exponential sum defined in \eqref{eq2.S}. To analyze the last integral, we partition the torus into major and minor arcs. Let $Q = L^C$, where $C > 0$ is a sufficiently large constant to be described later. We set $\mathfrak M = \mathfrak{M}(N^k, Q)$ and $\mathfrak m = \mathfrak {m}(N^k,Q)$. 

\subsection{The minor arc contribution}
\label{sec3.1}

The minor arc contribution to the integral \eqref{eq:HL integral} will be part of the error term in the \approximationformula. Let 
\[ \widehat{E_1}(\bm\xi;\lambda) = R(\lambda)^{-1}\int_{\mathfrak m} F(\theta; \bm\xi) e(-\lambda \theta) \, d\theta. \]
Since $R(\lambda) \gtrsim N^{n-k}$ for $\lambda \in \Gammakn$, the estimate \eqref{eq:dyadic_error_bound} for $\widehat{E_1}$ will follow from Lemma \ref{lem:L2 bound} if we show that for any $B>1$, we have 
\begin{equation}\label{minor.1} 
\int_{\mathfrak m} \sup_{\bm\xi \in \mathbb T^n}  |F(\theta; \bm\xi)| \, d\theta \lesssim_B N^{n-k}L^{-B}. 
\end{equation}
When $\theta \in \mathfrak m$, it has a rational approximation $a/q$ such that $Q \le q \le N^kQ^{-1}$, $(a,q)=1$ and $|q\theta - a| < q^{-1}$. By Lemma \ref{lem:Harman} with $f(x) = \theta x^k + \xi x$, we have
\begin{equation}\label{minor.2}
\sup_{(\theta, \xi) \in \mathfrak m \times \mathbb T}  |S_N( \theta, \xi) | \lesssim NQ^{-2^{1-2k}}L^{c_k},
\end{equation}
where $c_k$ is the constant in the statement of Lemma~\ref{lem:Harman}. Using this bound and H\"older's inequality, we get
\[ \int_{\mathfrak m} \sup_{\bm\xi \in \mathbb T^n}  |F(\theta; \bm\xi)| \, d\theta \lesssim NQ^{-2^{1-2k}}L^{c_k} \int_{\mathbb T} \sup_{\xi \in \mathbb T}  |S_N(\theta; \xi)|^{n-1} \, d\theta. \]
Hence, when $n \ge n_1(k)$, we obtain from Lemma \ref{lem:maxBourgain} that 
\[ \int_{\mathfrak m} \sup_{\bm\xi \in \mathbb T^n}  |F(\theta; \bm\xi)| \, d\theta \lesssim N^{n-k}Q^{-2^{1-2k}}L^{n+c_k}. \]
We can therefore choose $C_1 = C_1(B,k,n) > 0$ such that when $C \ge C_1$ in the definition of $Q$, the last inequality yields \eqref{minor.1}.

\subsection{The major arc contribution, I}

Let $R = Q^3$ and define 
\[ \mathfrak R = \mathfrak M(N, R), \quad \mathfrak N = \mathfrak M(N, Q), \quad \mathfrak r = \mathfrak m(N, R), \quad \mathfrak n = \mathfrak m(N, Q). \] 
We will show that when $\bm\xi \notin \mathfrak{N}^n$, the contribution of the major arcs $\mathfrak{M}$ to the integral \eqref{eq:HL integral} can be estimated similarly to the minor arc contribution.

Suppose that $\theta \in \mathfrak M_{a/q}$ and write $\delta = \theta - a/q$. Then, by partial summation,
\begin{align}\label{minor.3}
  |S_N(\theta, \xi)| &\le \sum_{h \in \mathbb U_q} \bigg| \sum_{\substack{p \le N\\p \equiv h \bmod q}} e(\delta p^k + \xi p) \bigg| + q \notag\\
  &\lesssim q(1 + N^k|\delta|) \sup_{M,h} \bigg| \sum_{\substack{p \le M\\p \equiv h \bmod q}} e(\xi p) \bigg|,
\end{align}
where the supremum is over $2 \le M \le N$ and $h\in \mathbb U_q$. When $\xi \in \mathfrak r$, it has a rational approximation $b/r$ such that 
\begin{equation}
R \le r \le NR^{-1}, \quad (b,r)=1, \quad |r\xi - b| \le RN^{-1}.
\end{equation} 
Hence, we may use Lemma \ref{lem:Balog} to show that
\begin{equation}\label{minor.4a}
\sup_{\theta \in \mathfrak M}  |S_N( \theta, \xi) | \lesssim R^{-1/2}NQL^3 \lesssim NQ^{-1/3}.
\end{equation}
On the other hand, if $\xi \in \mathfrak{R}_{b/r}$ for some major arc in $\mathfrak R$, Lemma \ref{lem:major arc approx} yields
\[ S_N(\theta,\xi) = g(a,q; b,r) I_N(\delta, \eta) + O\big( NQ^{-10} \big), \]
where $\eta = \xi - b/r$. When $\xi \notin \mathfrak N$, we have either $r \ge Q$ or $r|\eta| \ge QN^{-1}$. When $r \ge Q$, \eqref{eq2.3} yields 
\[ g(a,q; b,r) \lesssim Q^{-1/2+\eps}, \]
and when $r \le Q$ and $r|\eta| \ge QN^{-1}$, \eqref{eq2.3} and \eqref{eq2.5} yield
\[ g(a,q; b,r)I_N(\delta, \eta) \lesssim r^{-1/2+\eps}(N/|\eta|)^{1/2} \lesssim NQ^{-1/2+\eps}. \]
We conclude that inequality \eqref{minor.4a} holds whenever $\xi \notin \mathfrak{N}$. 

Thus, unless $\bm\xi \in \mathfrak{N}^n$, we have the bound \eqref{minor.4a} for some exponential sum $S_N(\theta, \xi_j)$. Using that bound in place of \eqref{minor.2} in the argument of \S\ref{sec3.1}, we conclude that when $C \ge C_2(B, n, k)$ in the definition of $Q$, the estimate \eqref{eq:dyadic_error_bound} holds for 
\[ \widehat{E_2}(\bm\xi; \lambda) = R(\lambda)^{-1}\Psi(\bm\xi)\int_{\mathfrak M} F(\theta; \bm\xi) e(-\lambda \theta) \, d\theta, \]
where $\Psi(\bm\xi)$ is any bounded function that is supported outside $\mathfrak N^n$. In particular,  the above inequality holds for 
\[ \Psi(\bm\xi) = 1 - \sum_{1 \le \mathbf q \le Q} \sum_{\mathbf a \in \mathbb U_{\mathbf q}} \psi_{N/Q}(\bfq\bm\xi - \bfa), \]
where $\psi$ is the bump function appearing in the statement of the \approximationformula.

\subsection{The major arc contribution, II}
\label{sec:3.3}

We now proceed to approximate the contribution of the major arcs to \eqref{eq:HL integral} when $\bm\xi$ lies close to $\mathfrak{N}^n$. For vectors $\bfa, \bfq$ with $1 \le \bfq \le Q$ and $\bfa \in \mathbb U_{\bfq}$, let $\bm{\mathfrak N}_{\bfa/\bfq}$ denote the support of $\psi_{N/Q}(\bfq\bm\xi - \bfa)$, and let $\bm{\mathfrak N}$ denote the union of all the different sets $\bm{\mathfrak N}_{\bfa/\bfq}$. Suppose that $\bm\xi = (\xi_1, \dots, \xi_n) \in \bm{\mathfrak N}_{\bfa/\bfq}$. When $\theta \in \mathfrak{M}_{a/q}$, we write $\delta = \theta - a/q$ and $\eta_j = \xi_j - a_j/q_j$. By Lemma \ref{lem:major arc approx}, 
\[ S_N(\theta,\xi_j) = g(a,q; a_j,q_j) I_N(\delta, \eta_j) + O\big( NQ^{-20} \big). \]
Since the major arcs are disjoint, we may define the function
\[ F^*(\theta; \bm\xi) = \prod_{j=1}^n g(a,q; a_j,q_j) I_N(\delta, \eta_j) \]
on all of $\mathfrak M \times \bm{\mathfrak N}$. This function satisfies
\[ \sup_{(\theta, \bm\xi) \in \mathfrak{M} \times \bm{\mathfrak N}} |F(\theta; \bm\xi) - F^*(\theta; \bm\xi)| \lesssim N^nQ^{-20}. \]
Since $|\mathfrak{M}| \lesssim QN^{-k}$, we can use the above inequality and Lemma \ref{lem:L2 bound} to show that \eqref{eq:dyadic_error_bound} holds for the error term
\[ \widehat{E_3}(\bm\xi; \lambda) = R(\lambda)^{-1} \sum_{1 \le \mathbf q \le Q} \sum_{\mathbf a \in \mathbb U_{\mathbf q}} \psi_{N/Q}(\bfq\bm\xi - \bfa)\int_{\mathfrak M} \big[ F(\theta; \bm\xi) - F^*(\theta; \bm\xi) \big] e(-\lambda \theta) \, d\theta. \]

By \eqref{eq:HL integral} and the above analysis, we have
\begin{equation}\label{eq:approx omega1}  
\ZFT{\omega_\lambda}(\bm\xi) = R(\lambda)^{-1}\sum_{1 \le \mathbf q \le Q} \sum_{\mathbf a \in \mathbb U_{\mathbf q}} \psi_{N/Q}(\bfq\bm\xi - \bfa) \int_{\mathfrak M} F^*(\theta; \bm\xi) e(-\lambda \theta) \, d\theta + \widehat{E_4}(\bm\xi; \lambda), 
\end{equation}
with an error term $\widehat{E_4}(\bm\xi; \lambda)$ that satisfies \eqref{eq:dyadic_error_bound}. Next, let 
\[ \mathfrak M' = \bigcup_{q \le Q} \bigcup_{a \in \mathbb U_q} \big\{ \theta \in \mathbb T : |\theta - a/q| \le QN^{-k} \big\}. \]
We want to extend the integral on the right side of \eqref{eq:approx omega1} to the set $\mathfrak{M'}$. The hypothesis on $n$ implies readily that $n \ge 3k$. We now apply once again Lemma \ref{lem:L2 bound} together with the inequality 
\begin{align*} 
\int_{\mathfrak M' \setminus \mathfrak M} \sup_{\bm\xi \in \bm{\mathfrak N}} |F^*(\theta; \bm\xi)| \, d\theta &\lesssim \sum_{q \le Q} \sum_{1 \le a \le q} q^{-n/2+\eps} \int_{Q/(qN^k)}^{\infty} \frac {N^n \, d\delta}{(1 + N^k \delta)^{n/k}} \\
&\lesssim Q^{2-n/k+\eps}N^{n-k} \lesssim Q^{-1+\eps}N^{n-k}, 
\end{align*}
where we have used \eqref{eq2.3} and \eqref{eq2.4}. Combining these estimates and \eqref{eq:approx omega1}, we obtain
\[ \ZFT{\vonMangoldtmeasure_\la}(\bm\xi) = R(\lambda)^{-1}\sum_{1 \le \mathbf q \le Q} \sum_{\mathbf a \in \mathbb U_{\mathbf q}} \psi_{N/Q}(\bfq\bm\xi - \bfa) \int_{\mathfrak M'} F^*(\theta; \bm\xi) e(-\lambda \theta) \, d\theta + \widehat{E_5}(\bm\xi; \lambda), \]
with an error term $\widehat{E_5}(\bm\xi; \lambda)$ that satisfies \eqref{eq:dyadic_error_bound}.

We now identify
\begin{equation}\label{eq:F-star}
\int_{\mathfrak M'} F^*(\theta; \bm\xi) e(-\lambda \theta) \, d\theta
\end{equation}  
as an integral over a subset of $\mathbb Q \times \mathbb R$ with respect to the product measure $\mu(r,\delta) = \nu(r) \times d\delta$, where $\nu$ is the counting measure on $\mathbb Q$ and $d\delta$ is the Lebesgue measure on $\mathbb R$.
Then one final appeal to Lemma \ref{lem:L2 bound} allows us to replace \eqref{eq:F-star} by 
\begin{equation}\label{eq:major-final} 
\sum_{q = 1}^{\infty} \sum_{a \in \mathbb U_q} \int_{\R} \bigg\{ \prod_{j=1}^n g(a,q; a_j,q_j) I_N(\delta, \eta_j) \bigg\} e(-\lambda(a/q + \delta)) \, d\delta. 
\end{equation}
This step requires an estimate for the quantity
\begin{equation}\label{eq:F-star-prime} 
\bigg\{ \sum_{\substack{a,q\\q > Q}} \int_{\R} + \sum_{a,q} \int_{|\delta| \ge QN^{-k}} \bigg\} \sup_{\bm\xi \in \bm{\mathfrak N}} \prod_{j=1}^n |g(a,q; a_j,q_j) I_N(\delta, \eta_j)| \, d\delta.
\end{equation}
Using \eqref{eq2.3} and \eqref{eq2.4}, we can bound the quantity \eqref{eq:F-star-prime} by 
\[ \sum_{q > Q} q^{-2}\int_{\R} \frac {N^n \, d\delta}{(1 + N^k|\delta|)^{n/k}} + \sum_{q=1}^{\infty} q^{-2}\int_{QN^{-k}}^{\infty} \frac {N^n \, d\delta}{(1 + N^k\delta)^{n/k}} \lesssim Q^{-1}N^{n-k}. \] 
We remark that the integral \eqref{eq:major-final} equals
\( \singseries\mathfrak I_{\lambda}(\bm\eta), \)
where
\[ \mathfrak I_{\lambda}(\bm\eta) = \int_{\R} \bigg\{ \prod_{j=1}^n I_N(\delta, \eta_j) \bigg\} e(-\lambda\delta) \, d\delta. \]
Hence,
\begin{equation}\label{eq:approx omega2}  
\ZFT{\vonMangoldtmeasure_\lambda}(\bm\xi) = R(\lambda)^{-1}\sum_{1 \le \mathbf q \le Q} \sum_{\mathbf a \in \mathbb U_{\mathbf q}} \singseries \psi_{N/Q}(\bfq\bm\xi - \bfa) \mathfrak I_{\lambda}(\bm\xi - \bfa/\bfq) + \widehat{E_\lambda}(\bm\xi), 
\end{equation}
an error term $\widehat{E_\lambda}(\bm\xi)$ that satisfies \eqref{eq:dyadic_error_bound}. To complete the proof of Theorem \ref{thm:approximation_formula}, we note that by the discussion on p. 498 in \cite{Stein93} (see also \S3.1 in \cite{Hughes_Vinogradov}), one has
\begin{align}\label{eq:major arcs 3}
\mathfrak{I}_\lambda(\bm\eta) &= \int_{\R} \int_{\R^n} \mathbf 1_{[0,N]^n}(\mathbf x)e(\bm\eta \cdot \mathbf x) e(\delta(\mathfrak{f}(\mathbf x) - \lambda)) \, d\mathbf x d\delta \notag\\
&= \lambda^{n/k-1} \int_{\R^n} \mathbf 1_{[0,N]^n}(\mathbf x)e(\bm\eta \cdot \mathbf x) \, d\sigma_\lambda(\mathbf x) 
=: \widetilde{d\sigma_{\lambda}}(\bm\eta), 
\end{align}
since the surface measure $d\sigma_{\lambda}$ is supported in the cube $[0,N]^n$.

\subsection{Remarks on the proof of \eqref{eq:sup_error_bound}}

We now take a moment to substantiate our claim, made in the introduction, that the $L^\infty$-bound \eqref{eq:sup_error_bound} holds under a weaker assumption on the dimension. The key observation is that to prove \eqref{eq:sup_error_bound} one does not need to refer to Lemma~\ref{lem:L2 bound}, and therefore, the inequality 
\begin{equation}\label{minor.1A} 
\int_{\mathfrak m} |F(\theta; \bm\xi)| \, d\theta \lesssim N^{n-k}L^{-B} 
\end{equation}
can replace \eqref{minor.1} in the treatment of the minor arcs. We can now use \eqref{minor.2} and the trivial majorization
\[ \int_{\T} |S_N(\theta, \xi)|^{2s} \, d\theta \lesssim L^{2s} \int_{\T} \bigg| \sum_{m \le N} e(\theta m^k) \bigg|^{2s} \, d\theta \]
to deduce \eqref{minor.1A} from the results in \S6 of \cite{Bourgain_hua}, provided that $k \ge 5$ and $n \ge n_0(k)$. When $k \le 4$, the same conclusion follows a sharp form of Hua's lemma, such as Lemma 1 in Vaughan~\cite{Vaughan_hua}. 

\section{Estimation of the main term contribution}
\label{sec:main_term}

In this section, we consider the maximal function of the convolution operator whose multiplier is the main term in the approximation formula. Given a sufficiently large $\lambda \in \Gammakn$, let $j$ be the unique integer such that $2^{j-1} \le \lambda < 2^{j}$. Let $M_{\lambda}$ denote the convolution operator with Fourier multiplier 
\[ \widehat{M_\lambda}(\bm\xi) = \sum_{q=1}^\infty \sum_{a\in \mathbb U_q} e\left( -\lambda a/q \right) \sum_{\bfq \leq Q} \widehat{M_\lambda^{a/q; \mathbf q}}(\bm\xi), \]
where 
\[ \widehat{M_\lambda^{a/q; \mathbf q}}(\bm\xi) :=  \sum_{\mathbf a \in \mathbb U_{\bfq}} \bigg\{ \prod_{i=1}^\dimension g(a, q; a_i, q_i) \bigg\} \psi_{N/Q}(\bfq\bm\xi-\bfa) \widetilde{d \sigma_{\lambda}}(\bm\xi-\mathbf a/\mathbf q), \]
with $N = 2^{j/k}$, $Q = (\log N)^C$ for some large fixed $C>0$. We write $M_*$ for the maximal operator defined pointwise as 
\[ M_*f(\mathbf x) := \sup_{\lambda \in \Gammakn} |M_\lambda f(\mathbf x)|. \]
Our main objective in this section is to prove the following theorem.

\begin{thm}\label{thm:main_term}
Let $\degree \geq 2$. If $\dimension \geq \max\{ 5, \degree/2+2 \}$ and $p > \frac{\dimension}{\dimension-2}$, then the maximal operator $M_*$ is bounded on $\ell^p(\Z^\dimension)$.
\end{thm}

\begin{remark}
Note that $\dimension_1(\degree),\dimension_2(\degree) \geq \degree/2+2$ so that these restrictions on the dimension $\dimension$ dominate in Theorem~\ref{mainmaxfunction}. 
In terms of the exponent $p$, our range of $\ell^p$-spaces is independent of the degree $\degree \geq 2$ and match those of the quadratic case (when $\degree=2$) for the integral spherical maximal function of Magyar, Stein and Wainger \cite{MSW}. 
In contrast, from \cite{Hughes_Vinogradov} we know that the integral $\degree$-spherical maximal functions of Magyar \cite{Magyar_dyadic} are unbounded on $L^p(\R^\dimension)$ for $p \leq \frac{\dimension}{\dimension-\degree}$ for each $\degree \geq 3$. The difference is that in our current setup the analytic piece of the operator (see below) is more localized in Fourier space than it is in previous works; this improves its boundedness properties. 
\end{remark}
 
If $\mathcal D \subset \R_+^n$, we introduce the maximal functions 
\[ M^{a/q;\mathcal D}_*f(\mathbf x) := \sup_{\lambda \in \Gammakn} \bigg| \sum_{\substack{\bfq \le Q\\ \bfq \in \mathcal D}} M^{a/q; \bfq}_\lambda f(\mathbf x) \bigg|, \]
so that we have the pointwise inequality
\begin{equation}\label{eq:mainterm:triangle_inequality}
M_*f(\mathbf x) \leq \sum_{q=1}^\infty \sum_{a\in \mathbb U_q} \sum_{\mathbf j \in \Z_+^n} M^{a/q;\mathcal D_{\mathbf j}}_*f(\mathbf x), 
\end{equation}
where $\mathcal D_{\mathbf j} = \big\{ \mathbf x \in \R^n : 2^{j_i-1} \le x_i < 2^{j_i}, \; 1 \le i \le n \big\}$. Applying the triangle inequality on $\ell^p(\Z^\dimension)$ in \eqref{eq:mainterm:triangle_inequality}, we see that 
\begin{equation}\label{eq:mainterm:1} 
\| M_*f \|_{\ell^p(\Z^\dimension)} \leq \sum_{q=1}^\infty \sum_{a\in \mathbb U_q} \sum_{\mathbf j \in \Z_+^n} \big\| M^{a/q;\mathcal D_{\mathbf j}}_*f \big\|_{\ell^p(\Z^\dimension)}. 
\end{equation}

Next, we estimate $\big\| M^{a/q;\mathcal D}_*f \big\|_{\ell^p(\Z^\dimension)}$ for a fixed rational number $a/q$ and a dyadic box $\mathcal D$. Suppressing the dependence on $a/q$, we write $M_\lambda^{\bfq}$ for the convolution operator $M_\lambda^{a/q;\bfq}$. Similarly to \cite{Bourgain_maximal_ergodic,MSW}, we first decompose each Fourier multiplier $\widehat{M_\lambda^{\mathbf q}}$ into an analytic piece and an arithmetic piece. Let $\psi$ be the bump function from the statement of the \approximationformula. For $\bfq \in \Z_+^\dimension$, we define the function $\Psi_{\bfq}(\bm\xi) = \psi(16\bfq\bm\xi)$ and note that, when $\lambda$ is large and $\bfq \le Q$, one has 
\[ \psi_{N/Q}(\bfq\bm\xi-\bfa) = \psi_{N/Q}(\bfq\bm\xi-\bfa)\Psi_\bfq(\bfq\bm\xi-\bfa). \] 
We also write
\[ G(\bfa) = G(a,q; \bfa, \bfq) := \prod_{i=1}^n g(a, q; a_i, q_i). \] 
We now define the Fourier multipliers
\begin{gather*}
\widehat{S^{\bfq}}(\bm\xi) := \sum_{\mathbf a \in \mathbb U_{\bfq}} G(a,q; \bfa, \bfq) \Psi_{\bfq}(\bfq\bm\xi-\bfa), \\
\widehat{T_\lambda^{\bfq}}(\bm\xi) := \sum_{\bfa \in \Z^n} \psi_{N/Q}(\bfq\bm\xi-\bfa) \widetilde{d\sigma}_{\lambda}(\bm\xi-\mathbf a/\mathbf q), 
\end{gather*}
so that 
\[ \widehat{M_\lambda^{\bfq}}(\bm\xi) = \widehat{T_\lambda^{\bfq}}(\bm\xi) \widehat{S^{\bfq}}(\bm\xi) . \]
Hence, 
\begin{equation}\label{eq:mainterm:2} 
\left\| M^{a/q;\mathcal D}_*f \right\|_{\ell^p(\Z^\dimension)} \le \sum_{\bfq \in \mathcal D} \big\| T_*^\bfq(S^\bfq f) \big\|_{\ell^p(\Z^\dimension)},
\end{equation}
where the maximal function $T_*^\bfq$ is defined by
\[ T^{\bfq}_*f(\mathbf x) := \sup_{\lambda \in \Gammakn} |T^{\bfq}_\lambda f(\mathbf x)|. \]

The estimation of the sum on the right side of \eqref{eq:mainterm:2} is broken into three lemmas. First, we note that when $\bfq \le Q$, the supports of the functions $\psi_{N/Q}(\bfq\bm\xi-\bfa)$ are disjoint, which puts the multipliers $T_\lambda^\bfq$ and $T_*^\bfq$ into the form considered by Magyar, Stein and Wainger in Section~2 of \cite{MSW}. In particular, Corollary 2.1 in \cite{MSW} allows us to transfer the bound in the next lemma to the maximal operators $T_*^{\bfq}$. (There is a technical difference in that our $\bfq$ is composed of different $q_i$; that is, in Magyar--Stein--Wainger they consider $\bfq=(q,\dots,q)$ whereas we are considering more general $\bfq$ where often $q_i \neq q_j$ for $i \neq j$. This however does not present a problem as we apply the Magyar--Stein--Wainger transference principle in each variable separately).

\begin{lemma}\label{lemma:analytic_part}
If $\dimension \ge \degree/2+2$ and $p > 1$, the maximal operator 
\[ T_*f(\mathbf x) := \sup_{\lambda \in \mathbb N} |f \star (\check{\psi_{\lambda^{1/\degree}(\log \lambda)^{-C}}} \star d\sigma_{\lambda})(\mathbf x)| \]
is bounded on $L^p(\R^\dimension)$.
\end{lemma}
\noindent The proof of this lemma appears in the appendix.

From this lemma and Corollary 2.1 in \cite{MSW}, we deduce that
\[ \big\| T_*^\bfq(S^\bfq f) \big\|_{\ell^p(\Z^\dimension)} \lesssim \big\| S^\bfq f \big\|_{\ell^p(\Z^\dimension)}. \]
Thus, \eqref{eq:mainterm:2} yields
\begin{equation}\label{eq:mainterm:2a}
\left\| M^{a/q;\mathcal D}_*f \right\|_{\ell^p(\Z^\dimension)} \lesssim \sum_{\bfq \in \mathcal D} \big\| S^\bfq f \big\|_{\ell^p(\Z^\dimension)}.
\end{equation}
Note that Corollary 2.1 in \cite{MSW} requires an appropriate choice of Banach spaces in order to apply it, hence our chosen decomposition of the multiplier and the application of their Corollary 2.1 at this point in the proof.

\begin{lemma}\label{lemma:main_term:ell2}
Let $\mathcal D$ be either a dyadic box of the form $\mathcal D_{\mathbf j}$ above or a singleton in $\Z_+^n$. Then for all $a,q$ and $\eps>0$, one has
\begin{equation}
\sum_{\bfq \in \mathcal D} \big\| S^\bfq f \big\|_{\ell^2(\Z^\dimension)} \lesssim_\eps q^{\eps-\dimension/2} \bigg\{ \sum_{\bfq \in \mathcal D} w_q(\mathbf q)^{2-\eps} \bigg\}^{1/2}\| f \|_{\ell^2(\Z^n)}, 
\end{equation}
where
\[ w_q(\bfq) = \prod_{i=1}^\dimension \frac {(q,q_i)}{q_i} .\]
\end{lemma}

\begin{lemma}\label{lemma:main_term:ell1}
For all $a,q,\bfq$ and $\eps>0$, one has
\begin{equation}\label{eq:main_term:ell1}
\big\| S^\bfq f \big\|_{\ell^1(\Z^\dimension)} \lesssim_\eps q^{\eps} w_q(\mathbf q)^{1-\eps}\| f \|_{\ell^1(\Z^n)}. 
\end{equation}
\end{lemma}

Now, we will use the lemmas to complete the proof of Theorem \ref{thm:main_term}; we prove Lemmas \ref{lemma:main_term:ell2} and~\ref{lemma:main_term:ell1} later in the section and Lemma \ref{lemma:analytic_part} in Appendix \ref{appendix_A}. First, we note that when $1 < p < 2$, interpolation between Lemma \ref{lemma:main_term:ell1} and the singleton case of Lemma \ref{lemma:main_term:ell2} yields
\begin{equation}\label{eq:mainterm:3}
\big\| S^\bfq f \big\|_{\ell^p(\Z^\dimension)} \lesssim_\eps q^{\eps-n/p'} w_q(\mathbf q)^{1-\eps}\| f \|_{\ell^p(\Z^n)},  
\end{equation}
where $p'$ is the conjugate exponent of $p$, defined by the relation $1/p + 1/p' = 1$. Using \eqref{eq:mainterm:2a} and \eqref{eq:mainterm:3}, we obtain 
\begin{equation}\label{eq:mainterm:4} 
\left\| M^{a/q;\mathcal D}_*f \right\|_{\ell^p(\Z^\dimension)} \lesssim \sum_{\bfq \in \mathcal D} \big\| S^\bfq f \big\|_{\ell^p(\Z^\dimension)} \lesssim_\eps q^{\eps-n/p'} \bigg\{ \sum_{\bfq \in \mathcal D} w_q(\mathbf q)^{1-\eps} \bigg\} \| f \|_{\ell^p(\Z^\dimension)} 
\end{equation}
for all $p > 1$. On the other hand, using \eqref{eq:mainterm:2a} and Lemma \ref{lemma:main_term:ell2}, we have 
\begin{equation}\label{eq:mainterm:5} 
\left\| M^{a/q;\mathcal D}_*f \right\|_{\ell^2(\Z^\dimension)} \lesssim \sum_{\bfq \in \mathcal D} \big\| S^\bfq f \big\|_{\ell^2(\Z^\dimension)} \lesssim_\eps q^{\eps-n/2} \bigg\{ \sum_{\bfq \in \mathcal D} w_q(\mathbf q)^{2-\eps} \bigg\}^{1/2} \| f \|_{\ell^2(\Z^\dimension)}. 
\end{equation}
When $1 < p < 2$, we can interpolate between \eqref{eq:mainterm:5} and \eqref{eq:mainterm:4} with $p_1 = (p + 1)/2$. If $\theta$ is defined so that $1/p = (1-\theta)/p_1 + \theta/2$, we get
\begin{equation}\label{eq:mainterm:6} 
\left\| M^{a/q;\mathcal D}_*f \right\|_{\ell^p(\Z^\dimension)} \lesssim_\eps q^{\eps-n/p'}\Sigma_1^{1-\theta}\Sigma_2^{\theta} \| f \|_{\ell^p(\Z^\dimension)}, 
\end{equation}
where
\[ \Sigma_s = \bigg\{ \sum_{\bfq \in \mathcal D} w_q(\mathbf q)^{s-\eps} \bigg\}^{1/s}. \]

Recall that we are interested in the case when $\mathcal D$ is the Cartesian product of intervals $[2^{j_i-1}, 2^{j_i})$, $j_i \in \mathbb Z_+$, and write $D_i = 2^{j_i}$. We have
\[ \Sigma_s^s \le \prod_{i=1}^n \bigg\{ \sum_{d \mid q} d^{s-\eps} \sum_{\substack{r \eqsim D_i\\ d \mid r}}r^{-s+\eps} \bigg\} \lesssim_\eps \prod_{i=1}^n \bigg\{ \sum_{\substack{d \mid q\\ d \le D_i}}(d/D_i)^{s-1-\eps} \bigg\}. \]
Hence, by the well-known inequality $\tau(q) \lesssim_\eps q^{\eps}$, 
\[ \Sigma_1 \lesssim_\eps (qD_1\cdots D_n)^{\eps} \]
and 
\[ \Sigma_2 \lesssim_\eps (qD_1\cdots D_n)^{\eps} \prod_{i=1}^n \left( \frac q{q+D_i} \right)^{1/2} =: (qD_1\cdots D_n)^{\eps}\Pi(q, \mathcal D). \]
Applying these bounds to the right side of \eqref{eq:mainterm:6}, we finally obtain
\begin{equation}\label{eq:mainterm:7} 
\left\| M^{a/q;\mathcal D}_*f \right\|_{\ell^p(\Z^\dimension)} \lesssim_\eps q^{2\eps-n/p'}(D_1\cdots D_n)^\eps \Pi(q, \mathcal D)^{\theta} \| f \|_{\ell^p(\Z^\dimension)}, 
\end{equation}
provided that $p > 1$.

We now apply \eqref{eq:mainterm:7} to all boxes $\mathcal D_{\mathbf j}$ that appear on the right side of \eqref{eq:mainterm:1} and then sum the resulting bounds over $\mathbf j$ to find that
\begin{align}\label{eq:mainterm:8} 
\sum_{\mathbf j \in \Z_+^n} \left\| M^{a/q;\mathcal D_{\mathbf j}}_*f \right\|_{\ell^p(\Z^\dimension)} \lesssim_\eps q^{2\eps-n/p'} \bigg\{ \sum_{j=1}^\infty \frac {2^{j\eps}q^{\theta/2}}{(q + 2^j)^{\theta/2}} \bigg\}^n \| f \|_{\ell^p(\Z^\dimension)}.
\end{align}
Let $j_0 = j_0(q)$ be the unique index for which $2^{j_0} \le q < 2^{j_0+1}$ and note that \eqref{eq:mainterm:8} is uniform in $a \in \mathbb U_q$. By splitting the series over $j$ at $j_0$, we deduce that 
\begin{align}\label{eq:mainterm:9} 
\sum_{a \in \mathbb U_q} \sum_{\mathbf j \in \Z_+^n} \left\| M^{a/q;\mathcal D_{\mathbf j}}_*f \right\|_{\ell^p(\Z^\dimension)} &\lesssim_\eps q^{1-n/p'+2\eps} \bigg\{ \sum_{j \le j_0} 2^{j\eps} + q^{\theta/2}\sum_{j > j_0} 2^{j(\eps-\theta/2)} \bigg\}^n \| f \|_{\ell^p(\Z^\dimension)} \notag\\
& \lesssim_\eps q^{1-n/p'+2\eps} 2^{nj_0\eps} \| f \|_{\ell^p(\Z^\dimension)} \lesssim_\eps q^{1-n/p'+2n\eps} \| f \|_{\ell^p(\Z^\dimension)},
\end{align}
provided that $0 < \eps < \theta/2$. After choosing $\eps > 0$ sufficiently small, Theorem \ref{thm:main_term} is an immediate consequence of \eqref{eq:mainterm:1} and \eqref{eq:mainterm:8}, provided that $n/p' > 2$, that is, $p > \frac{\dimension}{\dimension-2}$.

\subsection{Proofs of the lemmas}

\begin{proof}[Proof of Lemma \ref{lemma:main_term:ell2}]
Note that the functions $\Psi_{\bfq}(\bfq\bm\xi-\bfa)$ with distinct central points $\bfa/\bfq$, where $\bfq \in \mathcal D$, have disjoint supports. Indeed, if $\Psi_{\bfq'}(\bfq'\bm\xi-\bfa') \Psi_{\bfq''}(\bfq''\bm\xi-\bfa'') \ne 0$, with $\bfa'/\bfq' \ne \bfa''/\bfq''$, then for some index $i$, $1 \le i \le n$, we have
\[ \frac 1{4^{j_i}} < \frac 1{q_i'q_i''} \le \bigg| \frac {a_i'}{q_i'} - \frac {a_i''}{q_i''} \bigg| \le \bigg| \frac {a_i'}{q_i'} - \xi_i \bigg| + \bigg| \frac {a_i''}{q_i''} - \xi_i \bigg| \le \frac 1{8(q_i')^2} + \frac 1{8(q_i'')^2} \le \frac 1{4^{j_i}}; \]
a contradiction. Hence, Plancherel's theorem gives
\begin{align}\label{lem10:1}
\| S^{\bfq} f \|_{\ell^2(\Z^\dimension)}^2 
&= \Big\| \widehat{S^{\bfq} f} \Big\|_{L^2(\T^\dimension)}^2 
= \sum_{\bfa \in \mathbb U_{\bfq}} |G(\bfa)|^2 \int_{\T^n} \Psi_\bfq( \bfq\bm\xi - \bfa)^2 \big| \hat f(\bm\xi) \big|^2 \, d\bm\xi \notag\\
&\lesssim \Big( \max_{\bfa \in \mathbb U_{\bfq}} |G(\bfa)|^2 \Big)  \int_{\T^n} \Phi_\bfq(\bm\xi) \big| \hat f(\bm\xi) \big|^2 \, d\bm\xi,
\end{align}  
where
\[ \Phi_\bfq(\bm\xi) = \sum_{\bfa \in \mathbb U_\bfq} \Psi_\bfq( \bfq\bm\xi - \bfa). \]

Applying Lemmas~\ref{lem:Shparlinski} and \ref{lem:g-sum} to each factor $g(a, q; a_i, q_i)$ in $G(\bfa)$, we find that 
\begin{equation}\label{lem10:2}
|G(\bfa)| \lesssim_{\eps} q^{\eps-n/2} \prod_{i=1}^n \varphi \bigg( \frac {q_i}{(q,q_i)} \bigg)^{-1} \lesssim_\eps  q^{\eps-n/2} w_q(\bfq)^{1-\eps}, 
\end{equation}
where we have used the well-known inequality 
\begin{equation}\label{lem10:3} 
\varphi(m)^{-1} \lesssim m^{-1}\log\log m.  
\end{equation}

Combining \eqref{lem10:1}, \eqref{lem10:2} and Cauchy's inequality (in $\bfq$), we obtain
\begin{align*}
\sum_{\bfq \in \mathcal D} \| S^{\bfq}f \|_{\ell^2(\Z^\dimension)} &\lesssim_\eps q^{\eps-n/2} \bigg\{ \sum_{\bfq \in \mathcal D} w_q(\bfq)^{2-2\eps} \bigg\}^{1/2} \bigg\{  \int_{\T^n} \bigg( \sum_{\bfq \in \mathcal D}\Phi_\bfq(\bm\xi) \bigg) \big| \hat f(\bm\xi) \big|^2 \, d\bm\xi \bigg\}^{1/2} \\
&\lesssim_\eps q^{\eps-n/2} \bigg\{ \sum_{\bfq \in \mathcal D} w_q(\bfq)^{2-\eps} \bigg\}^{1/2} \big\| \hat f \big\|_{L^2(\T^n)},
\end{align*} 
by our earlier observation about the supports of the functions $\Psi_{\bfq}(\bfq\bm\xi-\bfa)$.
The lemma follows by another appeal to Plancherel's theorem. 
\end{proof}

\newcommand{\bfd}{{\bf d}}
\newcommand{\bft}{{\bf b}}
\newcommand{\bfr}{{\bf r}}
\newcommand{\bfy}{{\bf y}}
\newcommand{\bfz}{{\bf z}}
\newcommand{\fbq}{f_{\mathbf b, \mathbf q}}

\begin{proof}[Proof of Lemma \ref{lemma:main_term:ell1}]
For $\bft, \bfq \in \Z^n$ and $f: \Z^n \to \C$,  let $\fbq$ denote the restriction of $f$ to the residue class $\bft$ modulo $\bfq$ in $\Z^n$: i.e., $\fbq(\mathbf x) = f(\bft + \bfq\mathbf x)$. We remark that it suffices to prove the lemma for functions $\fbq$. Indeed, if the inequality 
\[ \| S^{\bfq} \fbq \|_{\ell^1(\Z^\dimension)} 
\le M \|\fbq\|_{\ell^1(\Z^\dimension)} \]
holds for all restrictions $\fbq$, then also
\[ \| S^{\bfq} f \|_{\ell^1(\Z^\dimension)} =
\sum_{\bft \in \Z_\bfq} \| S^{\bfq} \fbq \|_{\ell^1(\Z^\dimension)} \le M \sum_{\bft \in \Z_\bfq} \| \fbq \|_{\ell^1(\Z^\dimension)} = M \|f\|_{\ell^1(\Z^\dimension)}. \]

We now proceed to establish \eqref{eq:main_term:ell1} for restrictions $\fbq$. Note that 
\[ \widehat{\fbq}(\bm\xi+\bfa/\bfq) = e(\bft \cdot \bfa/\bfq) \cdot \widehat{\fbq}(\bm\xi). \]
From this we can deduce that 
\[ S^{\bfq} \fbq(\bfy) = H(a,q; \bfy - \bft, \bfq) \big( \widecheck{\Psi_{\bfq}} \star \fbq \big) (\bfy), \] 
where $\widecheck{\Psi_{\bfq}}$ denotes the inverse Fourier transform of $\Psi_\bfq(\bfq\bm\xi)$ and 
\[ H(a,q; \mathbf u, \bfq) = \sum_{\bfa \in \mathbb U_\bfq} G(a,q; \bfa, \bfq) e(-\mathbf u \cdot \bfa/\bfq ). \]
(Note that $H(a,q; \mathbf u, \bfq)$ is a multidimensional version of the sum $h(a,q; u, r)$ that appears in the proof of Lemma \ref{lem:g-sum}.) We now have
\[ \| S^{\bfq} \fbq \|_{\ell^1(\Z^\dimension)} = \sum_{\bfy \in \Z^\dimension} \big| H(a,q; \bfy - \mathbf b, \bfq) \big( \widecheck{\Psi_{\bfq}} \star \fbq \big) (\bfy) \big|. \]
We rearrange the last sum according to the residue class of $\bfy$ modulo $\bfq$. Since $H(a,q; \bfy - \mathbf b, \bfq)$ depends only on the residue class of $\bfy$ modulo $\bfq$, we get
\begin{align}
\| S^{\bfq} \fbq \|_{\ell^1(\Z^\dimension)} &= \sum_{\bfr \in \Z_\bfq} \big| H(a,q; \bfr - \bft, \bfq) \big| \sum_{\bfz \in \Z^\dimension} \big| \big( \widecheck{\Psi_{\bfq}} \star \fbq \big) (\bfq\bfz + \bfr) \big| \notag\\
&= \sum_{\bfr \in \Z_\bfq} \big| H(a,q; \bfr - \bft, \bfq) \big| \sum_{\bfz \in \Z^\dimension} \bigg| \sum_{\mathbf x \in \Z^\dimension} \widecheck{\Psi_{\bfq}}(\bfq\bfz+\bfr-\mathbf x) \fbq(\mathbf x) \bigg| \notag\\
& \leq \sum_{\bfr \in \Z_\bfq} \big| H(a,q; \bfr - \bft, \bfq) \big|\sum_{\mathbf x \in \Z^\dimension} |\fbq(\mathbf x)| \sum_{\bfz \in \Z^\dimension} \big| \widecheck{\Psi_{\bfq}}(\bfq\bfz+\bfr-\mathbf x) \big|. \label{eq:Sfbq}
\end{align} 

The sum over $\mathbf z$ on the right side of \eqref{eq:Sfbq} is $\bfq$-periodic in $\mathbf r - \mathbf x$, so we may assume that $-\frac 12 \le (\mathbf r - \mathbf x)/\mathbf q \le \frac 12$. 
Since $\widecheck{\Psi_\bfq}(\mathbf m) = \widehat{\Psi_\bfq}(\mathbf m)$, we find that 
\begin{align*}
\sum_{\mathbf z \in \mathbb Z^n}  |\widecheck{\Psi_\bfq}(\mathbf q\mathbf z + \mathbf r - \mathbf x)| &= \sum_{\mathbf z \in \mathbb Z^n}  \bigg| \int_{\mathbb R^n} \psi_{\bfq}(16\bfq\bm\xi) e((\mathbf q\mathbf z + \mathbf r - \mathbf x) \cdot \bm\xi) \, d\bm\xi| \\
&= \sum_{\mathbf z \in \mathbb Z^n}  \frac 1{q_1^2 \cdots q_n^2} \bigg| \widehat{\psi_{16}} \bigg( \frac {\mathbf z + (\mathbf r - \mathbf x)/\bfq}{\mathbf q} \bigg) \bigg| 
\\
&\lesssim 
\frac 1{q_1^2 \cdots q_n^2} \sum_{\mathbf z \in \mathbb Z^n} \frac 1{1 + | (\mathbf z + (\mathbf r - \mathbf x)/\bfq)/\mathbf q|^{2n}}
\lesssim \frac 1{q_1 \cdots q_n}.
\end{align*}

Inserting the last bound into the right side of \eqref{eq:Sfbq}, we deduce the estimate 
\begin{align*}
\| S^{\bfq}\fbq \|_{\ell^1(\Z^\dimension)} \lesssim \frac {\| \fbq \|_{\ell^1(\Z^\dimension)}}{q_1 \cdots q_\dimension} \sum_{\bfr \in \Z_\bfq} \big| H(a,q; \bfr - \bft, \bfq) \big|.
\end{align*}
Since
\[ \sum_{\bfr \in \Z_\bfq} \big| H(a,q; \bfr - \bft, \bfq) \big| = \sum_{\mathbf u \in \Z_\bfq} \big| H(a,q; \mathbf u, \bfq) \big| = \prod_{j=1}^n \bigg\{ \sum_{u \in \Z_{q_i}} \big| h(a, q; u, q_i) \big| \bigg\}, \] 
Lemma \ref{lem:g-sum}(iii) now yields
\[ \| S^{\bfq}\fbq \|_{\ell^1(\Z^\dimension)} \lesssim \| \fbq \|_{\ell^1(\Z^\dimension)} \prod_{j=1}^n \bigg( \frac {\tau(q_i)}{\varphi(q_i/(q,q_i))} \bigg). \]
The desired estimate follows from \eqref{lem10:3} and the bound $\tau(m) \lesssim_\eps m^{\eps}$.
\end{proof}

\newcommand{\integralaverage}{B}
\newcommand{\primeaverage}{A}

\section{Comparison with the integral maximal function}
\label{section:integral_comparison}

In this section, we show that the maximal function of the error term is bounded on $\ell^p(\Z^n)$ for a range of $p$ by comparing the averages $\primeaverage_\lambda$ for $\lambda \in \Gammakn$ with the bounds for the corresponding integral operators. This combined with the boundedness of the main term shows that the maximal function $A_*$ is bounded on $\ell^p(\Z^\dimension)$. As we will see, our range of $\ell^p$-boundedness for the averages $\primeaverage_*$ matches that of the integral maximal function $\integralaverage_*$ below, possibly up to endpoints. 

For $f : \Z^\dimension \to \C$ and $\mathbf x \in \Z^\dimension$, define the integral averages by 
\[ \integralaverage_\lambda f(\mathbf x) := (f \star \sigma_\lambda)(\mathbf x) = \frac {1}{\#\{ \bfy \in \Z_+^\dimension : \form(\bfy)=\lambda \}} \sum_{\form(\bfy) = \lambda} f(\mathbf x-\bfy), \]
along with their maximal function 
\[ \integralaverage_* f(\mathbf x) := \sup_{\lambda \in \N} |\integralaverage_\lambda f(\mathbf x)|. \]
The operator $\integralaverage_*$ is equivalent to Magyar--Stein--Wainger's discrete spherical maximal function. Our goal is to prove the following comparison between the integral maximal function and the Waring--Goldbach maximal function. 

\begin{thm}\label{thm:integral_comparison}
Suppose that $1 < p_0 < 2$ and $\dimension \ge n_1(k)$. If $\integralaverage_*$ and $M_*$ are bounded on $\ell^{p_0}(\Z^\dimension)$, then $\primeaverage_*$ is bounded on $\ell^p(\Z^\dimension)$ for $p>p_0$. 
\end{thm}

\begin{proof}
Recall from the \approximationformula\ that for each $\lambda \in \Gammakn$ we have 
\[ \primeaverage_\lambda f(\mathbf x) = M_\lambda f(\mathbf x)+E_{\lambda}f(\mathbf x). \]
We will use the decay of the dyadic maximal function of the error term on $\ell^2(\Z^\dimension)$. By \eqref{eq:dyadic_error_bound}, we have 
\begin{equation}\label{error:thm1} 
\Big\| \sup_{\lambda \eqsim 2^j} |E_\lambda f| \Big\|_{\ell^2(\Z^\dimension)} \lesssim j^{-K} \|f\|_{\ell^2(\Z^\dimension)} 
\end{equation}
for an arbitrarily large, fixed $K>0$, provided that the parameter $C$ in Theorem \ref{thm:approximation_formula} is chosen sufficiently large. Our first order of business is to establish the following matching bound on $\ell^{p_0}(\Z^n)$:  
\begin{equation}\label{error:blowup}
\Big\| \sup_{\lambda \eqsim 2^j} |E_\lambda f| \Big\|_{\ell^{p_0}(\Z^\dimension)} 
\lesssim 
j^\dimension \|f\|_{\ell^{p_0}(\Z^\dimension)}.
\end{equation}

For each $\mathbf x \in \Z^\dimension$ we have
\[ |\primeaverage_\lambda f(\mathbf x)| \lesssim 
(\log \lambda)^\dimension (\integralaverage_\lambda |f|)(\mathbf x). \]
Thus, 
\[ |E_\lambda f(\mathbf x)| \lesssim |M_\lambda f(\mathbf x)| + (\log \lambda)^\dimension (\integralaverage_\lambda |f|)(\mathbf x) \]
for each $\lambda \in \Gammakn$ and all $\mathbf x \in \Z^\dimension$. In turn, 
\[ \sup_{\lambda \eqsim 2^j} |E_\lambda f(\mathbf x)| \lesssim \sup_{\lambda \eqsim 2^j} |M_\lambda f(\mathbf x)| + j^\dimension \sup_{\lambda \eqsim 2^j} (\integralaverage_\lambda |f|)(\mathbf x). \]
Taking $\ell^{p_0}(\Z^\dimension)$ norms and applying the hypotheses, we deduce \eqref{error:blowup}.

For $p_0 < p < 2$, let $\theta$ be such that $1/p = (1-\theta)/p_0 + \theta/2$, and then choose $K$ sufficiently large to ensure that $n(1 - \theta) - K\theta \le -2$. Then interpolation between \eqref{error:thm1} and \eqref{error:blowup} reveals that 
\[ \Big\| \sup_{\lambda \eqsim 2^j} |E_\lambda f| \Big\|_{\ell^p(\Z^\dimension)} \lesssim j^{-2} \|f\|_{\ell^p(\Z^\dimension)}. \]
Summing over $j \in \N$, we find that 
\[ \Big\| \sup_{\lambda \in \Gammakn} |E_\lambda f| \Big\|_{\ell^p(\Z^\dimension)} \lesssim \|f\|_{\ell^p(\Z^\dimension)} \]
for all $p_0<p<2$. Combining this with our hypothesis that $M_*$ is bounded on $\ell^{p_0}(\Z^\dimension)$ (and hence, also on $\ell^p(\Z^\dimension)$---by interpolation with the trivial $\ell^\infty(\Z^\dimension)$ bound), we are done. \end{proof}

\begin{proof}[Proof of Theorem \ref{mainmaxfunction}]
For $\degree=2$, the main theorem of \cite{MSW} shows that $B_*$ is bounded on $\ell^p(\Z^\dimension)$ for $p > \frac{\dimension}{\dimension-2}$ and $\dimension \geq 5$. 
For $\degree \geq 3$, Theorem~1 of \cite{Hughes_restricted} we have that $B_*$ is bounded on $\ell^p(\Z^\dimension)$ for $p> \max\{\frac{\dimension}{\dimension-\degree}, 1+\frac{\degree^2}{2(\dimension-\degree[\degree+2])+\degree^2}, 1+\frac{\degree}{\frac{2\dimension}{\degree[\degree-1]}-\degree}\}$ and $\dimension \geq \max\{ \degree(\degree+2),\degree^2(\degree-1)\}$. Thus the theorem is true for $p> 1+\frac{\degree^2[\degree-1]}{2\dimension-\degree^2[\degree-1]} = \frac{2\dimension}{2\dimension-\degree^2[\degree-1]}$ and $\dimension \geq \degree^2(\degree-1)$. 
\end{proof}

\section{Applications}
\label{sec:6}

\newcommand{\measure}{\mu} 
\newcommand{\measurespace}{X} 
\newcommand{\measurespacepoint}{x} 
\newcommand{\mfxn}{f} 
\newcommand{\transferfxn}{F} 

\newcommand{\primelattice}{\mathbb{P}^\dimension} 

\newcommand{\measurepreservingtransform}{T} 

\newcommand{\inbraces}[1]{\left\{{#1}\right\}} 
\newcommand{\eof}[1]{e\left({#1}\right)} 

\newcommand{\ZdFT}[1]{\widehat{#1}} 
\newcommand{\contFT}[1]{\widetilde{#1}} 

\renewcommand{\T}{\mathbb{T}} 
\newcommand{\unitsmod}[1]{\mathbb{Z}(#1)^\times } 
\newcommand{\bigmodulus}{Q} 
\renewcommand{\bigradius}{R} 
\renewcommand{\radius}{\lambda^{1/\degree}} 

\newcommand{\Lpnorm}[3]{\left\| #3 \right\|_{L^{#1}(#2)}}
\newcommand{\norm}[3]{\left\| #3 \right\|_{{#1}(#2)}}
\newcommand{\innerproductof}[2]{\left( #1,#2 \right)}

\newcommand{\spectralpoint}{{\bm\eta}} 

\newcommand{\bumpfxn}{\varphi} 

\newtheorem*{oscillation_inequality_mps}{Oscillation inequality for measure preserving systems}
\newtheorem*{oscillation_inequality_lattice}{Transferred oscillation inequality}

\newcommand{\average}{A} 
\newcommand{\ergodicaverage}{\mathcal{A}} 
\newcommand{\acceptableradii}{\Gamma_{n,k}}
\newcommand{\primepoint}{{\bf p}} 
\newcommand{\integralpoint}{{\bf m}} 
\newcommand{\discretecube}[1]{\mathcal{C}(#1) } 
\newcommand{\indicator}[1]{{\bf 1}_{#1}} 

\newcommand{\HLmultiplier}{m}  
\newcommand{\HLmultiplieronarc}{m^{a/q,\bfa/\bfq}}  
\newcommand{\transferHLoponarc}{\mathcal{M}^{\bfq}}  
\newcommand{\transfermaxHLoponarc}{\mathcal{M}^{\bfq}_*}  
\newcommand{\smoothHLmultiplier}{n}  
\newcommand{\smoothHLmultiplieronarc}{n^{a/q,\bfa/\bfq}}  
\newcommand{\smoothtransferHLoponarc}{\mathcal{N}^{a/q,\bfa/\bfq}}  
\newcommand{\transfererror}{\mathcal{E}}

\renewcommand{\Q}{\mathbb{Q}}

In this section, we prove Theorems~\ref{thm:decayatirrationalfrequencies} and \ref{thm:pointwise_ergodic}. Recall that in what follows, $(X,\mu)$ denotes a probability space with a commuting family of invertible measure preserving transformations $T=(T_1,...,T_\dimension)$ without any rational points in their spectrum. For a function $f: X \to \C$ the Waring--Goldbach ergodic averages on $X$ with respect to $T$ for $\lambda \in \Gammakn$ are defined by \eqref{ergavg}. 

\subsection{Proof of Theorem~\ref{thm:decayatirrationalfrequencies}}

Fix $\eps > 0$ and let $\delta > 0$ be a parameter to be chosen later (in terms of $\eps$). Since $\bm\xi \notin \Q^n$, we may assume without loss of generality that $\xi_1 \notin \Q$. Then, we can choose a convergent $b/r$ to the continued fraction of $\xi_1$ with $r > 2\delta^{-1}$. 

Now, for a large $\lambda \in \Gammakn$, let $N = \lambda^{1/k}$ and $Q = (\log N)^C$, where $C = C(1) > 0$ is the power in the \approximationformula\ corresponding to having \eqref{eq:sup_error_bound} for $B = 1$. We note that for sufficiently large $\lambda$, there is at most one rational point $\bfa/\bfq$ such that
\begin{equation}\label{eq:6.1}
1 \le \bfq \le Q, \quad \bfa \in \mathbb U_{\bfq}, \quad \psi_{N/Q}(\bfq\bm\xi - \bfa) > 0. 
\end{equation} 
If such a rational point does not exist, the main term in \eqref{eq:approximation_formula} vanishes, and we have 
\[ \widehat{\omega_\lambda}(\bm\xi) \lesssim (\log\lambda)^{-1}. \] 
Otherwise, \eqref{eq:approximation_formula} yields
\[ \widehat{\omega_\lambda}(\bm\xi) \lesssim \big| \singseries \widetilde{d\sigma_1}(N(\bm\xi - \bfa/\bfq)) \big| + (\log\lambda)^{-1}, \]
where $\bfa/\bfq$ satisfies \eqref{eq:6.1}. Using \eqref{eq2.3} with $\eps = 1/(4n)$, we deduce that, for $n \ge 5$,
\[ \singseries \lesssim q_1^{-9/20} \sum_{q = 1}^{\infty} q^{21/20-n/2}(q,q_1)^{1/2} \lesssim q_1^{-9/20} \sum_{d \mid q_1} d^{1/2} \sum_{\substack{q = 1\\ d \mid q}}^{\infty} q^{21/20-n/2} \lesssim q_1^{-2/5}. \]
Hence, 
\[ \widehat{\omega_\lambda}(\bm\xi) \lesssim q_1^{-2/5} \big| \widetilde{d\sigma_1}(N(\bm\xi - \bfa/\bfq)) \big| + (\log\lambda)^{-1}. \]

Using the decay of $\widetilde{d\sigma_1}$ (see for example \cite{BNW}), we may now choose $\delta$ so that 
\[ q_1^{-2/5} \widetilde{d\sigma_1}(N(\bm\xi - \bfa/\bfq)) \lesssim \eps, \]
unless
\begin{equation}\label{eq:6.2}
1 \le q_1 \le \delta^{-1} \quad \text{and} \quad |\xi_1 - a_1/q_1| \le (\delta N)^{-1}. 
\end{equation} 
Thus, we have
\[ \widehat{\omega_\lambda}(\bm\xi) \lesssim \eps + (\log\lambda)^{-1}, \]
unless $a_1,q_1$ and $\xi_1$ satisfy \eqref{eq:6.2}. To complete the proof of the theorem, we will show that for sufficiently large $\lambda$, inequalities \eqref{eq:6.2} are inconsistent with the choice of $b/r$. 

Suppose that conditions \eqref{eq:6.2} do hold and recall that $|r\xi_1 - b| < r^{-1}$. Then
\[ |bq_1 - a_1r| \le \frac {rq_1}{\delta N} + \frac {q_1}r < \frac r{\delta^2N} + \frac 12 < 1, \]
as $N \to \infty$. Since $b/r$ and $a_1/q_1$ are reduced fractions, we conclude that $a_1 = b$ and $q_1 = r$. The latter, however, contradicts the inequalities $q_1 \le \delta^{-1} < r/2$. \qed

\begin{remark}
We comment that a shorter proof of Theorem~\ref{thm:decayatirrationalfrequencies} exists by using the decay of the error term in \eqref{eq:dyadic_error_bound}, but this proof has the advantage of not relying on the bound \eqref{eq:dyadic_error_bound} and instead uses \eqref{eq:sup_error_bound}.
\end{remark}

\subsection{The Pointwise Ergodic Theorem}

\newcommand{\largenum}{K}
\newcommand{\transF}{F}

To prove Theorem~\ref{thm:pointwise_ergodic} we will utilize the Calder\'on transference principle and in doing so, we need to introduce some notation. Let $\largenum$ be a large natural number and define the discrete cube 
\[ \discretecube{\largenum} := \big\{ \integralpoint \in \Z^\dimension : |m_i| \leq \largenum \text{ for } i=1, \dots, \dimension \big\}. \]
For a $\mu$-measurable function \( f : X \to \C \), define its truncated \emph{transfer function}, 
\[ \transF(\measurespacepoint,\integralpoint) = f(T^{\integralpoint} x) \indicator{\discretecube{N}}(\integralpoint). \]  
For \( \lambda \in \Gammakn \), also define the \emph{transferred averages}
\begin{equation*}
\ergodicaverage_{\lambda} \transF(x,\integralpoint) 
:= 
\frac{1}{R(\lambda)} \sum_{\form({\bf p}) = \lambda} \log({\bf p}) \transF(x,\integralpoint+{\bf p})
\end{equation*}
and their \emph{tail maximal function} 
\begin{equation*}
\ergodicaverage_{>R} \transF(x,\integralpoint) 
:= 
\sup_{\lambda > R} \left| \ergodicaverage_{\lambda} \transF(x,\integralpoint) \right| 
.
\end{equation*}

We endow the transfer space $X \times \Z^\dimension$ with the product measure of $\mu$ on $X$ and the counting measure on $\Z^\dimension$. As in \cite{Hughes_Vinogradov}, we deduce Theorem~\ref{thm:pointwise_ergodic} from the tail oscillation inequality below. We refer to \cite{Hughes_Vinogradov} for the details of this reduction, which relies on the Calder\'on transference principle.

\begin{prop}[Transferred Oscillation Inequality]
Let $\mfxn$ be a bounded function of mean zero on $\measurespace$ and $\transF$ its transfer function. 
For each $\eps>0$, there exists a sufficiently large radius $R = R(\eps,\mfxn)$ such that 
\begin{equation}\label{oscillation_inequality_for_lattice}
\Lpnorm{2}{\measurespace \times \Z^\dimension}{\ergodicaverage_{>R} \transF } < \eps \Lpnorm{2}{\measurespace \times \Z^\dimension}{\transF} .
\end{equation}
\end{prop}

The proof of the transferred oscillation inequality requires a few steps, which we carry out in succession. First, we extend the \approximationformula\ to the lifted averages. For $\bm\xi \in \T^\dimension$, define the partial $\Z^\dimension$-Fourier transform as 
\[
\ZdFT{\transF}(\measurespacepoint,\bm\xi) 
:= 
\sum_{\integralpoint \in \Z^\dimension} \transF(\measurespacepoint,\integralpoint) \eof{\integralpoint \cdot \bm\xi} 
.\] 
The reader may verify that 
\begin{equation}\label{eq:ergodic_averagesFT}
\ZdFT{\ergodicaverage_{\lambda} \transF} (x,\bm\xi) 
= \ZdFT{\omega_\lambda}(\bm\xi) \ZdFT{\transF}(x,\bm\xi).
\end{equation}
Equation \eqref{eq:ergodic_averagesFT} allows us to extend the multipliers on $\Z^\dimension$ to multiplers on $\measurespace \times \Z^\dimension$.  Suppressing the dependence on $a,q$, we define the convolution operators $\transferHLoponarc_\lambda$ by the multipliers 
\[
\ZdFT{\transferHLoponarc_\lambda \transferfxn}(\measurespacepoint, \bm\xi) := \widehat{M_\lambda^{\bfq}}(\bm\xi) \ZdFT{\transferfxn}(\measurespacepoint, \bm\xi),  
\]
where $\widehat{M_\lambda^{\bfq}}$ is the Fourier multiplier from Section \ref{sec:main_term} with $N = N_\lambda = \lambda^{1/k}$ and $Q = (\log\lambda)^C$. 
Similarly, define the error term by 
\begin{equation}
\ZdFT{\transfererror_\lambda \transferfxn}(\measurespacepoint, \bm\xi) 
= 
\ZdFT{E_\lambda}(\bm\xi) \ZdFT{\transferfxn}(\measurespacepoint, \bm\xi)
.\end{equation}
Also define their tail maximal functions similarly to $\ergodicaverage_{>R} \transferfxn$. 

Our estimates on the error term in Theorem~\ref{thm:approximation_formula} transfer over to show that 
\begin{equation}
\left\| \sup_{\lambda>R} |\transfererror_\lambda \transferfxn| \right\|_{L^2(X \times \Z^\dimension)} 
\lesssim 
(\log R)^{-B_1} \|\transferfxn\|_{L^2(X \times \Z^\dimension)}
\end{equation} 
for all large $B_1 > 0$, so that choosing $R$ sufficiently large we may make this arbitrarily small. This shows that the averages are equiconvergent with the main term. Lemmas~\ref{lemma:analytic_part} and \ref{lemma:main_term:ell2} and a version of \eqref{eq:mainterm:8} for $p=2$ combine to give 
\begin{align*}
\bigg\| \sup_{\lambda>R} \bigg| \sum_{q,a} \sum_{\bfq > Q} \transferHLoponarc_{\lambda} \transferfxn| \bigg\|_{L^2(X \times \Z^\dimension)} 
& \leq 
\sum_{a,q,\bfa,\bfq > Q} 
\| \transferHLoponarc_{>R} \transferfxn \|_{L^2(X \times \Z^\dimension)} 
\\
& \lesssim 
Q^{-C_2} \|\transferfxn\|_{L^2(X \times \Z^\dimension)}
\end{align*} 
for some positive $C_2$ when $\dimension \geq \max\{\dimension_1(\degree),\dimension_2(\degree)\}$. 

Our final proposition completes the proof of Theorem~\ref{thm:pointwise_ergodic}. 
This is the only place where the vanishing of the rational spectrum is used. 
\begin{prop}\label{proposition:ergodic_main_terms_small}
If $\epsilon>0$, then there exists a radius $R = R(\mfxn;\epsilon,\bigmodulus) \in \acceptableradii$ sufficiently large such that for all $q, \bfq \leq \bigmodulus$, $a \in U_{q}$ and $\bfa \in U_{\bfq}$, 
\begin{equation}\label{low_frequency_oscillation}
\Lpnorm{2}{\measurespace \times \Z^\dimension}{\transferHLoponarc_{>R} \transferfxn}
\lesssim 
\epsilon \Lpnorm{2}{\measurespace \times \Z^\dimension}{\transferfxn} 
\end{equation}
with implicit constants independent of $a, \bfa; q, \bfq$. 
\end{prop}

As this is the essential part, we include the proof. 
Our proof will follow that of Proposition~9.2 in \cite{Hughes_Vinogradov} for the integral $\degree$-spherical maximal function. 
Unlike the integral maximal function where the localizing bump function depends on the modulus $q$, our current localizing bump function depends on the radius so that the continuous part or the multiplier behaves like a smooth Hardy--Littlewood averaging operator. 
This simplifies our exposition.

\begin{proof}
By Lemma~\ref{lemma:analytic_part}, the tail maximal function of the multipliers 
\[ 
\psi_{\lambda^{1/k}(\log{\lambda})^{-C}}(\bfq\bm\xi-\bfa) \widetilde{d\sigma_{\lambda_0}}(N(\bm\xi-\bfa/\bfq)) 
\] 
is bounded on $L^2(X \times \Z^\dimension)$ with the bound 
\begin{align*}
\Lpnorm{2}{\measurespace \times \Z^\dimension}{\transferHLoponarc_{>R} \transferfxn} 
& \lesssim 
\Lpnorm{2}{\measurespace \times \Z^\dimension}{\left( \prod_{i=1}^\dimension g(a, q; a_i, q_i) \right) \check{\psi_{\bigradius'}^{\bfa/\bfq}} * \transferfxn}
\end{align*}
where $\bigradius' := \bigradius (\log\bigradius)^{-C}$. 
To prove Proposition~\ref{proposition:ergodic_main_terms_small} it suffices to show that 
\begin{equation}
\Lpnorm{2}{\measurespace \times \Z^\dimension}{\check{\psi_{\bigradius'}^{\bfa/\bfq}} * \transferfxn} 
\lesssim 
\epsilon \Lpnorm{2}{\measurespace \times \Z^\dimension}{\transferfxn}
\end{equation}
for each $\bfa,\bfq$ and sufficiently large $\bigradius$ depending on $\epsilon$. 
Plancherel's Theorem and the Spectral Theorem imply 
\begin{equation*}
\Lpnorm{2}{\measurespace \times \Z^\dimension}{\check{\psi_{\bigradius'}^{\bfa/\bfq}} * \transferfxn}^2 
= 
\int_{\T^\dimension} \int_{\T^\dimension} | \psi_{\bigradius'}(\bfq\bm\xi-\bfa) |^2 
\sum_{\integralpoint_1,\integralpoint_2 \in \discretecube{\largenum}} \eof{(\integralpoint_1-\integralpoint_2)[\spectralpoint+\bm\xi]} 
\, d\bm\xi \, d\nu_{\mfxn}(\spectralpoint)
.
\end{equation*}
Once again, see \cite{Magyar:ergodic} for this derivation. 
Collecting $\integralpoint_1-\integralpoint_2 = \integralpoint$, we define the sequence 
\[
\Delta_N(\integralpoint) 
:=  
\frac{\#\{(\integralpoint_1,\integralpoint_2) \in \discretecube{\largenum} \times \discretecube{\largenum} : \integralpoint_1-\integralpoint_2 = \integralpoint \}} {\#{\discretecube{\largenum}}} 
. 
\]
The above becomes 
\begin{equation*}
\Lpnorm{2}{\measurespace \times \Z^\dimension}{\check{\psi_{\bigradius}^{\bfa/\bfq}} * \transferfxn}^2 
= 
\int_{\T^\dimension} \int_{\T^\dimension} | \psi_{\bigradius'}(\bfq\bm\xi-\bfa)|^2 
\sum_{\integralpoint \in \Z^\dimension} \#{\discretecube{\largenum}} \Delta_\largenum(\integralpoint) 
\cdot e(\integralpoint \cdot [\bm\xi+\spectralpoint]) 
\, d\bm\xi \, d\nu_{\mfxn}(\spectralpoint) 
. 
\end{equation*}

Note that $\Delta_\largenum \to 1$ as $\largenum \to \infty$. 
This implies that $\ZdFT{\Delta_\largenum}(\bm\xi) \to \delta_0(\bm\xi)$ tends pointwise to the Dirac delta function on $\T^\dimension$ as $\largenum \to \infty$. 
Therefore, 
\begin{align*}
\#{\discretecube{\largenum}}^{-1} \Lpnorm{2}{\measurespace \times \Z^\dimension}{\check{\psi_{\bigradius}^{\bfa/\bfq}} * \transferfxn}^2 
& = 
\int_{\T^\dimension} \int_{\T^\dimension} | \psi_{\bigradius'}(\bfq\bm\xi-\bfa)|^2 
\sum_{\integralpoint \in \Z^\dimension} \Delta_\largenum(\integralpoint) 
\cdot e(\integralpoint \cdot [\bm\xi+\spectralpoint]) 
\, d\bm\xi \, d\nu_{\mfxn}(\spectralpoint) 
\\
& = 
\int_{\T^\dimension} \int_{\T^\dimension} | \psi_{\bigradius'}(\bfq\bm\xi-\bfa) |^2 
\cdot \ZdFT{\Delta_\largenum}(\bm\xi+\spectralpoint) 
\,\, d\bm\xi \, d\nu_{\mfxn}(\spectralpoint) \\
& = \int_{\T^\dimension} (|\psi_{\bfq\bigradius}(\cdot-\bfa/\bfq)|^2 * \ZdFT{\Delta_\largenum})(\spectralpoint) \; d\nu_{\mfxn}(\spectralpoint) 
\end{align*}
where the convolution is on the torus. 
Now we make use of the fact that multiplier is localized to low frequencies. 
For all $\epsilon>0$, there exists $\largenum_\epsilon \in \N$ such that $|\ZdFT{\Delta_{\largenum}} - \delta_0| < \epsilon$ for all $\largenum>\largenum_\epsilon$ and 
\begin{align*}
&\int_{\T^\dimension}
\left| |\psi_{\bfq\bigradius}(\cdot-\bfa/\bfq)|^2 * \ZdFT{\Delta_\largenum}
(\spectralpoint)
d\nu_{\mfxn}(\spectralpoint) \right| \\& \leq \int_{\T^\dimension} | |\psi_{\bfq\bigradius}(\cdot-\bfa/\bfq)|^2*|\ZdFT{\Delta_\largenum}-\delta_0|(\spectralpoint) |+ \left| |\psi_{\bfq\bigradius}(\cdot-\bfa/\bfq)|^2 * \delta_0(\spectralpoint) \right| d\nu_{\mfxn}(\spectralpoint) \\
&= \int_{\T^\dimension} |\psi_{\bfq\bigradius}(\cdot-\bfa/\bfq)|^2 * |\ZdFT{\Delta_\largenum}-\delta_0|(\spectralpoint) \; d\nu_{\mfxn}(\spectralpoint) + \int_{\T^\dimension} |\psi_{\bigradius}(\bfq\spectralpoint-\bfa)|^2 d\nu_{\mfxn}(\spectralpoint) \\ 
&\lesssim \epsilon \Lpnorm{2}{\measurespace}{\mfxn}^2 + \nu_{\mfxn}(|\spectralpoint-\bfa/\bfq| \lesssim |\bfq R|^{-1}) . 
\end{align*}
For $\bfa/\bfq = 0$, $\nu_{\mfxn}(|\spectralpoint| \lesssim |\bfq R|^{-1}) \to \nu_{\mfxn}(0)$ as $R \to \infty$, but $\nu_{\mfxn}(0) = |\int_{\measurespace} \mfxn d \measure|^2 = 0$. 
For $\bfa/\bfq \neq 0$, $\nu_{\mfxn}(|\spectralpoint-\bfa/\bfq| \lesssim |\bfq R|^{-1}) \to \nu_{\mfxn}(\bfa/\bfq)$ as $R \to \infty$, but $\nu_{\mfxn}(\bfa/\bfq) = 0$ by our assumption on the rational spectrum. 
Since there are finitely many $a/q$ and $\bfa/\bfq$, we can finish by choosing $R$ large enough.
Note that our parameter $R$ depends on the spectral measure $\nu_{\mfxn}$ and consequently on the function $\mfxn$, in addition to $\epsilon$ and $Q$. 
\end{proof}

\appendix

\section{Estimates for mollified continuous $k$-spherical averages}
\label{appendix_A}

In this appendix, we sketch the \(L^p(\R^\dimension)\)-boundedness of the maximal functions 
\[ T_*f(\mathbf x) := \sup_{\lambda \in \N} |T_\lambda f(\mathbf x)| \]
defined by the averages 
\[ T_\lambda f = (\widecheck{\psi_{N/Q}} \star  d\sigma_{\lambda}) \star f, \]
where $N = N_\lambda = \lambda^{1/k}$ and $Q$ satisfies $1 \le Q \le (\log\lambda)^C$ for some constant $C > 0$. In Section~\ref{sec:main_term} we applied the Magyar-Stein-Wainger transference principle to this maximal function to obtain \(\ell^p(\Z^\dimension)\)-bounds.

We will need the following two propositions in our proof. 
\begin{prop}\label{proposition:log_approximate_identity}
Let $K \in \mathbb N$. For all $\lambda>1$ and $Q \ge 1$, one has 
\begin{equation}
\widetilde{\psi_{N/Q}} \star d\sigma_{\lambda} (\mathbf x) \lesssim_K  \frac{QN^{-\dimension}}{(1+|\mathbf x|/N|)^K}. 
\end{equation} 
\end{prop}

\begin{proof}
By rescaling, we only need to prove that 
\[
\widetilde{\psi_{1/Q}} \star d\sigma_1 (\mathbf x) \lesssim_K \frac{Q}{(1+|\mathbf x|)^K}.
\]
This is well-known for the spherical measure (see for example, equation (5.5.12) in \cite{grafakos08a}), but there is essentially no difference in the proof for the remaining $\degree$-spherical measures when $\degree \geq 3$. 
\end{proof}



Let $P_j$ denote the smooth Littlewood--Paley projection operator adpated to frequencies of size approximately $2^j$; that is, 
\[\widetilde{P_jf}(\bm\xi) := (\psi(\bm\xi/2^{j+1})-\psi(\bm\xi/2^j))\widetilde{f}(\bm\xi).\] 
\begin{prop}
\label{LPdecay}
For $\dimension \geq 2$ and $\degree \geq 2$, we have that 
\begin{equation}\label{deq:l^2bound}
\left\|\sup_{\lambda\in\N} |P_{j}f \star (\widetilde{\psi_{N/Q}} \star d\sigma_{\lambda})| \right\|_{L^2(\R^n)}
\lesssim 
(1+2^j)^{\frac{1}{2}-\frac{\dimension-1}{\degree}} \|f\|_{L^2(\R^n)}
.
\end{equation}
\end{prop}

\begin{proof}
From \cite{Hughes_Vinogradov} we know that $\widetilde{d\sigma_\lambda}(\xi) \lesssim (1+|\lambda^{1/\degree}\xi|)^{-\frac{\dimension-1}{\degree}}$. 
This implies that $\widetilde{P_jd\sigma_\lambda}(\xi) \lesssim (1+|\lambda^{1/\degree}2^j|)^{-\frac{\dimension-1}{\degree}}$. 
The bound for the supremum now follows from this and a similar bound for the derivative via the Sobolev embedding theorem. 
\end{proof}

\begin{proof}[Proof of Lemma~\ref{lemma:analytic_part}]
Fix $C>0$. 
Since $T_*$ is trivially bounded $L^\infty(\R^\dimension)$, we only need to show that it is also bounded on $L^p(\R^\dimension)$ for all $1 < p \leq 2$. 
We introduce a frequency decomposition of our multipliers depending on a fixed parameter $\Lambda \gg 1$ to be chosen later:
\begin{equation*}
\widetilde{P^{\text{low}}f}(\xi) := \psi_{\Lambda^{1/\degree}(\log\Lambda)^{-C}}(\xi) \, \widetilde{f}(\xi),
\end{equation*}
and 
\begin{equation*}
P^{\text{high}}f := f-P^{\text{low}}f. 
\end{equation*}
With this decomposition if $\lambda>\Lambda$, then $(P^{\text{\text{low}}}f) \star \widetilde{\psi_{\lambda^{1/\degree}(\log \lambda)^{-C}}} = P^{\text{low}}f$. 
Proposition~\ref{proposition:log_approximate_identity} implies the following pointwise bound: 
\begin{equation}
\label{almostapprox}
\sup_{\lambda \leq \Lambda} | f \star (\widetilde{\psi_{\lambda^{1/\degree}(\log \lambda)^{-C}}} \star d\sigma_{\lambda}) (x) | 
\lesssim_C 
(\log \Lambda)^C Mf(x), 
\end{equation}
and we also have for all $\lambda\in\N$ the bound 
\begin{equation}
\label{almostapprox2}
\sup_{\lambda \in \N} | P^{\text{low}}f \star (\widetilde{\psi_{\lambda^{1/\degree}(\log \lambda)^{-C}}} \star d\sigma_{\lambda}) (x) | 
\lesssim_C 
(\log \Lambda)^C Mf(x) 
.
\end{equation}

We first prove a restricted weak-type inequality via interpolation, splitting up $|\{ T_* f > \alpha \}| $ into three sets where we use \eqref{almostapprox2}, \eqref{almostapprox}, and \eqref{deq:l^2bound}.  
Let $F \subset \R^\dimension$ and $f := {\bf 1}_F$ denote its indicator function so that 
\begin{align*}
|\{ T_* f > \alpha \}| 
& \leq 
|\{  \sup_{\lambda\in\N} | P^{\text{low}}f * (\widetilde{\psi_{\lambda^{1/\degree}(\log \lambda)^{-C}}} \star d\sigma_{\lambda}) | > \alpha/2 \}| 
\\ & \hspace{5mm} + |\{  \sup_{\lambda\in\N} | P^{\text{high}}f \star (\check{\psi_{\lambda^{1/\degree}(\log \lambda)^{-C}}} * d\sigma_{\lambda}) | > \alpha/2 \}| 
\\
& \leq 
|\{  \sup_{\lambda\in\N} | P^{\text{low}}f \star (\widetilde{\psi_{\lambda^{1/\degree}(\log \lambda)^{-C}}} \star d\sigma_{\lambda}) | > \alpha/2 \}| 
\\ & \hspace{5mm} + |\{  \sup_{\lambda\leq\Lambda} | P^{\text{high}}f \star (\widetilde{\psi_{\lambda^{1/\degree}(\log \lambda)^{-C}}} \star d\sigma_{\lambda}) | > \alpha/4 \}| 
\\ & \hspace{5mm} + |\{  \sup_{\lambda>\Lambda} | P^{\text{high}}f \star (\widetilde{\psi_{\lambda^{1/\degree}(\log \lambda)^{-C}}} \star d\sigma_{\lambda}) | > \alpha/4 \}| 
\\ 
& \lesssim 
(\log\Lambda)^C \|f\|_{1} \alpha^{-1} 
+ |\{  \sup_{\lambda>\Lambda} | P^{\text{high}}f \star (\widetilde{\psi_{\lambda^{1/\degree}(\log \lambda)^{-C}}} \star d\sigma_{\lambda}) | > \alpha/4 \}|. 
\end{align*}
Observe that for each $x \in \R^\dimension$, we have 
\begin{equation*}
\sup_{\lambda>\Lambda} | P^{\text{high}}f \star (\widetilde{\psi_{\lambda^{1/\degree}(\log \lambda)^{-C}}} \star d\sigma_{\lambda})(x) | 
\leq 
\sum_{j \geq \log_2 \Lambda} \left( \sup_{\lambda>2^j} |P_{j}f \star (\widetilde{\psi_{\lambda^{1/\degree}(\log \lambda)^{-C}}} \star d\sigma_{\lambda})(x)| \right)
\end{equation*}
so that 
\begin{align*}
|\{ T_* f > \alpha \}| 
& \lesssim 
(\log\Lambda)^C \|f\|_{1} \alpha^{-1} 
+ |\{  \sum_{j \geq \log_2 \Lambda} \left( \sup_{\lambda>2^j} |P_{j}f \star (\widetilde{\psi_{\lambda^{1/\degree}(\log \lambda)^{-C}}} \star d\sigma_{\lambda}) | \right) > \alpha/4 \}|
\\ 
& \lesssim 
(\log\Lambda)^C \|f\|_{1} \alpha^{-1} 
+ \alpha^{-2} \sum_{j \geq \log_2 \Lambda} (1+2^{j/\degree})^{\frac{1}{2}-\frac{\dimension-1}{\degree}} \|f\|_2^2
\end{align*}
by Chebychev's inequality and Proposition~\ref{LPdecay}. 
Therefore 
\begin{align*}
|\{ T_* f > \alpha \}| 
& \lesssim 
(\log\Lambda)^C \|f\|_{1} \alpha^{-1} + \Lambda^{\frac{1}{2\degree}-\frac{\dimension-1}{\degree^2}} \|f\|_{2}^2 \alpha^{-2}
\\ 
& = 
(\log\Lambda)^C |F| \alpha^{-1} + \Lambda^{\frac{1}{2\degree}-\frac{\dimension-1}{\degree^2}} |F| \alpha^{-2}.
\end{align*}
Here $|F|$ denotes the Lebesgue measure of the set $F$. 

To interpolate between $L^1$ and $L^2$ we need $\frac{1}{2\degree}-\frac{\dimension-1}{\degree^2}<0$ which occurs precisely when $\dimension > \degree/2+1$. 
For any $1 < p < 2$ we choose $\Lambda>0$ depending on $0 < \alpha \leq 1$ so that both summands are dominated by $|F| \alpha^{-p}$, which yields the restricted weak-type inequality. 
The Marcinkiewicz interpolation theorem gives the strong-type inequality. 
\end{proof}

\bibliographystyle{amsplain}

\end{document}